\documentclass[11pt]{article}
\usepackage{amssymb,amsmath,amsthm,secdot,color,bbm}
\usepackage[title]{appendix}
\DeclareMathOperator{\I}{\mathbbm{1}}%

\DeclareMathOperator{\Law}{Law}%

\DeclareMathOperator{\supp}{supp}

\usepackage[colorlinks=true]{hyperref}
\hypersetup{pdftex, colorlinks=true, linkcolor=Purple, citecolor=Purple,
 filecolor=blue,urlcolor=blue}
\def\E{\hskip.15ex\mathsf{E}\hskip.10ex}
\def\P{\mathsf{P}}

\def\eps{\varepsilon}

\let\temp\phi
\let\phi\varphi
\let\varphi\temp

\newtheorem{Theorem}{Theorem}[section]
\newtheorem{Lemma}[Theorem]{Lemma}
\newtheorem{Proposition}[Theorem]{Proposition}

\newtheorem{Assumption}[Theorem]{Assumption}

\newtheorem{Definition}[Theorem]{Definition}
\theoremstyle{definition}\newtheorem{Remark}[Theorem]{Remark}

\newcommand{\snorm}[1]{{\| #1
  \|}}

\numberwithin{equation}{section}

\renewcommand{\ge}{\geqslant}
\renewcommand{\le}{\leqslant}

\newcommand{\nn}{\nonumber}
\newcommand{\wt}{\widetilde}
\newcommand{\wh}{\widehat}

\newcommand{\C}{\mathcal{C}}
\newcommand{\D}{\mathbb{D}}
\newcommand{\F}{\mathcal{F}}
\newcommand{\SL}{\mathbb{L}}
\newcommand{\OL}{\mathcal{L}}
\newcommand{\N}{\mathbb{N}}
\newcommand{\R}{\mathbb{R}}
\renewcommand{\S}{\mathcal{S}}
\newcommand{\Z}{\mathbb{Z}}

\newcommand{\la}{\langle}
\newcommand{\ra}{\rangle}

\definecolor{Red}{rgb}{1,0,0}
\definecolor{Blue}{rgb}{0,0,1}
\definecolor{Olive}{rgb}{0.41,0.55,0.13}
\definecolor{Yarok}{rgb}{0,0.5,0}
\definecolor{Green}{rgb}{0,1,0}
\definecolor{MGreen}{rgb}{0,0.8,0}
\definecolor{DGreen}{rgb}{0,0.55,0}
\definecolor{Yellow}{rgb}{1,1,0}
\definecolor{Cyan}{rgb}{0,1,1}
\definecolor{Magenta}{rgb}{1,0,1}
\definecolor{Orange}{rgb}{1,.5,0}
\definecolor{Violet}{rgb}{.5,0,.5}
\definecolor{Purple}{rgb}{.75,0,.25}
\definecolor{Brown}{rgb}{.75,.5,.25}
\definecolor{Grey}{rgb}{.7,.7,.7}
\definecolor{Black}{rgb}{0,0,0}

\newcommand{\ignore}[1]{{}}

\begin{document}

\title{{Strong existence and uniqueness for stable stochastic differential equations with distributional drift}}

\author{ Siva Athreya%
  \thanks{Supported in part by ISF-UGC Grant and CPDA.}\setcounter{footnote}{3}
 \and
 Oleg Butkovsky%
 \thanks{Supported in part by ISF-UGC grant No. 1131/14 and DFG Research Unit FOR 2402.}
 \and
Leonid Mytnik%
 \thanks{Supported in party by ISF-UGC grant No. 1131/14.}
}

\maketitle










\begin{abstract}
We consider the stochastic differential equation
$$ dX_t = b(X_t) dt + dL_t,$$ where the drift $b$ is a generalized function and $L$ is a
symmetric one dimensional $\alpha$-stable L\'evy processes, $\alpha \in (1, 2)$. We
define the notion of solution to this equation and establish strong existence and uniqueness
whenever $b$ belongs to the Besov--H\"{o}lder space $\C^\beta$ for $\beta >1/2-\alpha/2$.
\end{abstract}

\section{Introduction}

In this article we consider the stochastic differential equation (SDE)
\begin{equation}\label{mainsde}
X_t=x+ \int_0^tb(X_s)\,ds + L_t,\quad t\ge0,
\end{equation}
where the initial condition $x\in\R$, $L$ is a symmetric
$1$-dimensional $\alpha$-stable process, $\alpha\in(1,2)$, and the
drift $b$ is in the H\"older-Besov space $\C^\beta =\C^\beta(\R,\R)$ for $\beta \in\R$ (see \cite [Definition~7]{P15}). When $\beta \le 0$ this equation is not well--posed in the classical sense. Indeed, in this case $b$ is not a function but just a distribution and the expression $b(X_s)$ is not well-defined. Thus it is not
clear a priori what should be called a \textit{solution} to the
SDE. Inspired by the Bass--Chen approach \cite{BC}, we formulate a
natural notion of a solution to \eqref{mainsde} (see Definition
\ref{D:sol}) and establish strong existence and pathwise uniqueness of
a solution when $\beta > \frac{1-\alpha}{2}$, see Theorem~\ref{T:1}.

It is well--known for quite a long time that ordinary differential equations (ODEs) regularize when an additional forcing by Brownian motion is added. Indeed, if $b\colon \R^d\to\R^d$, $d\ge1$, is a $\beta$-H\"older function, $0<\beta<1$, then an ODE
\begin{equation*}
d X_t=b(X_t) dt,\quad t\ge0
\end{equation*}
might have multiple solutions or no solutions when $b$ is a bounded measurable function. However, once the random forcing by Brownian motion $(B_t)_{t\ge 0}$ is added, the corresponding SDE
\begin{equation}
\label{sde_BM}
dX_t = b(X_t)dt +dB_t,\quad t\ge0
\end{equation}
has a unique strong solution even for bounded measurable $b$ without any additional assumptions on continuity. This phenomenon is called in the literature ``regularization by noise''.
For SDE \eqref{sde_BM} strong existence and uniqueness of solutions was established by Zvonkin in~\cite{zvonkin74} in case $d=1$ and extended by
Veretennikov \cite{ver80} to a multidimensional case. Later
Krylov and R\"ockner \cite{kr_rock05} generalized this result for the case of a
locally unbounded $b$ under a suitable integrability condition. In all the cases, the  proofs use a Zvonkin-type transformation \cite{zvonkin74} that allows to make the ``non-regular'' drift much  more regular. 

It turned out that in the one-dimensional case it is possible to consider drifts that are not functions but rather generalized functions. In this case one needs to specify what  is exactly meant by a solution to \eqref{sde_BM} since the term $\int_0^t b(X_s)\, ds$ is not well-defined. This was done by Bass and Chen in~\cite{BC}, who suggested a natural definition of a solution via an approximating scheme. They have also established strong existence and uniqueness for \eqref{sde_BM} whenever $b$ is the distributional derivative of $\C^{\gamma}$ functions with $\gamma>1/2$. Their main tool was again the Zvonkin method; they used the fact that for $d=1$ the Zvonkin transformation can completely eliminate the drift.

The above question has also been studied for other types of forcing
instead of  Brownian motion. For results about general continuous
forcings we refer the reader to \cite{CG16}. In the case of a forcing by a pure
jump process, it is clear that this process should have ``sufficiently
many'' small jumps. That is,  if  Brownian motion is replaced in
\eqref{sde_BM} just by the standard Poisson process, then this does
not give any improvement in the regularity properties of the
equation. Indeed, the equation will already have multiple solutions
while still ``waiting'' for the first jump of the Poisson
process. Thus it is natural to expect that the bigger  the intensity
of small jumps the rougher  drift $b$ can be.

Indeed, Tanaka, Tsuchiya, Watanabe in~\cite{bib:ttw74} proved that in the case $d=1$ equation \eqref{mainsde} has a pathwise unique solution if  $L$ is a symmetric $\alpha$-stable process, $b$ is a bounded continuous function and $\alpha\ge1$  (recall that the bigger the parameter $\alpha\in(0,2)$, the higher is the intensity 
of small jumps). On the other hand, it was also shown in \cite{bib:ttw74} that if $\alpha\in(0,1)$ and $b$ is bounded H\"older continuous with exponent $\beta$, where $0< \beta < 1- \alpha$, then equation \eqref{mainsde}
might have multiple solutions.  The case of higher dimensions was
resolved by Priola in~\cite{Pr12} who showed that in the case of
dimension $d\geq 2$ and $\alpha \in (1, 2)$, the pathwise uniqueness
holds for this equation if the drift $b$ is bounded and H\"older
continuous with exponent $\beta>1-\alpha/2$. This result was extended by Chen, Song and Zhang
in~\cite{bib:csz15} to the case $\alpha\in (0,1)$. Further, Bogachev
and Pilipenko in \cite[Theorem 1 and Remark 3]{BP15} showed strong existence
and uniqueness for (\ref{mainsde}) for $b$ belonging to a certain Kato class (see \cite[Definition 1, (9) and (10)]{BP15}). We note that the Kato class includes all bounded measurable functions
but does not necessarily contain $b$ which are generalized
functions; in particular, it is known that $\C^\beta$ for $\beta <0$ is not contained in the Kato class.


From the discussion above, the reader may notice the following gap
between the cases of $\alpha<2$ and $\alpha=2$. For $\alpha\in
(1,2)$ the strong existence and uniqueness for (\ref{mainsde}) is shown
by Bogachev and Pilipenko~\cite{BP15} for $b$ in the Kato class in any
dimension $d\geq 1$. However in the case of $\alpha =2$ and $d=1$
Bass and Chen \cite{BC} have shown that the strong existence and
uniqueness hold under much milder assumptions on $b$; namely $b$ can be the distributional derivative of $\C^{\gamma}$
 functions with $\gamma>1/2$.
This paper closes this gap, by showing that for $\alpha \in (1,2)$ in dimension
 $d=1$ the strong existence and uniqueness hold for (\ref{mainsde}) under much more relaxed conditions on the drift $b$ than in \cite{BP15}.
Our main result in Theorem~\ref{T:1} states that for $\alpha\in (1,2)$
there is a unique strong solution to~(\ref{mainsde}) if $b$ is in the
H\"older-Besov Space $\C^\beta$ for $\beta >
\frac{1-\alpha}{2}$. That is, loosely speaking, $b$ is allowed to be
a distributional derivative of a H\"older continuous function with the
H\"older exponent greater than $\frac{3-\alpha}{2}$. We note that
this bound on the regularity of $b$ exactly matches the result
of Bass, Chen~\cite{BC} for the case $\alpha=2$. To the best of our knowledge our
result is the first strong existence and uniqueness result for stable SDEs with a general distributional
drift.

To obtain this result we further develop the  Zvonkin drift transformation method. Note that in the case of  stable processes the Zvonkin transformation does not eliminate the drift even in $d=1$; thus the approach of Bass and Chen  \cite{BC} is not applicable here. An additional challenge comes from the fact that even with the proper definition the process $\int_0^t b(X_s)\,ds$ might be of infinite variation and hence $X$ might not be a semimartingale.


In this article we mainly consider \textit{strong} solutions to \eqref{mainsde}. Let us briefly mention that
other notions of existence and uniqueness have also been studied
for (\ref{sde_BM}) and (\ref{mainsde}). Weak existence and
uniqueness results for (\ref{sde_BM}) have been obtained in
~\cite{bib:frw03}, \cite{bib:frw04}, \cite{bib:fir17}, and
\cite{bib:zz17}. The question of weak uniqueness and existence for~\eqref{mainsde} was studied in Kulik \cite{K15} for $b$  measurable and locally bounded and in  Song~\cite{bib:kims14}, Chen, Wang~\cite{bib:chenw16} for $b$ from a certain Kato class. Some of these results are  also valid  for the case when
the SDEs have a non-trivial diffusion coefficient. A stronger notion of
path-by-path uniqueness has been established for (\ref{sde_BM}) in
a seminal work by Davie \cite{davie} when $b$ is a bounded
measurable function and it has been generalized by Priola \cite{Pr18} 
for (\ref{mainsde}) with $\alpha \in (1,2)$ and $b$ is a bounded continuous function with 
$\beta$ with $\beta > 1- \frac{\alpha}{2}.$

In the next section we will present the main results of the paper.

\bigskip

\noindent \textbf{Acknowledgments}. The authors are grateful to Nicolas Perkowski and David Pr\"{o}mel for helpful discussions about Besov spaces.
Part of the work on the project
has been done during the visits of the authors to the Fields Institute
(Toronto, Canada), Technion---Israel Institute of Technology (Haifa,
Israel) and Indian Statistical Institute (Bangalore). We thank them
all for their support and hospitality. OB is very grateful to the
Fields institute and especially to Bryan Eelhart for their support,
hospitality, and incredible coffee breaks. LM is grateful to the
Johannes Gutenberg University Mainz where part of this research has
been done.

\section{Main Result}\label{sec:mainresult}
We begin with introducing the basic notation and definitions.
For $k\in\Z_+$, $D\subset \R^k$ and function $f\colon D \to \R$, we denote its
supremum norm by $\snorm{f}:=\sup_{x\in D}|f(x)|$. If the
function $f$ is random, then supremum in the definition of $\snorm{f}$
will be taken only over \textit{nonrandom} variables. For $f,g\colon \R \to \R$, we define
$\langle f,g \rangle := \int_\R f(x)g(x)dx.$

 We denote by
$\C^\infty_b$ the space of all bounded continuous functions
$\R\to\R$. Let $\C^\infty_c$ be the space of all functions from
$\C^\infty_b$ with compact support. Let $\S$ be the space of Schwartz
functions $\R\to\R$ and let $\S^\prime$ be its dual space, i.e., the
space of Schwartz distributions. We will work with the Besov--H\"older
spaces ${\mathcal C}^\gamma:={\mathcal B}^\gamma_{\infty, \infty}$,
where $\gamma \in \R$, which are defined using the Littlewood-Paley
blocks (see, e.g., \cite [Definition~7]{P15}). Let $\|\cdot\|_\gamma$
be the norm associated with the space ${\mathcal C}^\gamma$,
$\gamma\in\R$.

We recall that for $\gamma\in(0,\infty)\setminus\mathbb{N}$ the space $\C^\gamma$ is just the usual H\"{o}lder space of functions that are $\lfloor\gamma\rfloor$ times continuously differentiable
and whose $\lfloor\gamma\rfloor$-th derivative is H\"{o}lder continuous with exponent
$\gamma-\lfloor\gamma\rfloor$. For $\gamma\in(-1,0)$ the space $\C^\gamma$ includes all derivatives (in the distributional sense) of  H\"{o}lder functions  with exponent $\gamma+1$.

In this article we study stochastic differential equation \eqref{mainsde}. Recall that since the drift $b$ is not a function but just a distribution the notion of the solution to this equation is not well--defined. Inspired by \cite[Definition~2.1]{BC} we give the following definition.

 \begin{Definition}\label{D:sol}
\label{defsolution} Let $\beta \in \R$, $\alpha \in (1,2)$ and $L=(L_t)_{t \geq 0}$ be a symmetric $1$-dimensional $\alpha$-stable process. 
We say that a c\`adl\`ag process $X=(X_t)_{t\ge0}$ is a
solution to \eqref{mainsde} with the initial condition $x\in\R$ if there exists a continuous process $A=(A_t)_{t\ge0}$ such that:
\begin{enumerate}
 \item $X_t=x+A_t+L_t$, $t\ge0$;
 \item for any sequence of functions $(b_n)_{n\in\Z_+}$ such that $b_n\in\C_b^\infty$, $n\in\Z_+$ and $\|b_n-b\|_{\beta}\to0$ as $n\to\infty$ we have
\begin{equation}\label{sxodka}
 A^n_t:= \int_0^tb_n(X_s)\,ds\to A_t,\,\,\text{as $n\to\infty$}
\end{equation}
 in probability uniformly over bounded time intervals;
\item for any $T>0$ and any $\kappa<(1+\frac{\beta}{\alpha})\wedge1$ there exists $C=C(T,\kappa)>0$ such that
\begin{equation}\label{3cond}
\E|A_t-A_s|^2\le C |t-s|^{2\kappa},\quad s,t\in[0,T].
\end{equation}

\end{enumerate}
\end{Definition}

Given a symmetric $\alpha$-stable process $L$ on a probability space,
a strong solution to (\ref{mainsde}) is a c\`adl\`ag process $X$ that
is adapted to the complete filtration generated by $L$ and which is a solution to
(\ref{mainsde}). A weak solution of (\ref{mainsde}) is a couple
$(X,L)$ on a complete filtered probability space $(\Omega,{\cal F}, (\F_t)_{t \ge 0}, \P)$ such that $X_t$ is adapted to ${\cal F}_t$, $L_t$ is
an $ (\F_t)_{t \ge 0}$ adapted symmetric $\alpha$-stable process and
$X$ is a solution to (\ref{mainsde}). We say weak uniqueness holds for
(\ref{mainsde}) if whenever $(X,L)$ and $(\wt{X}, \wt{L})$ are
two weak solutions of (\ref{mainsde}) and $X_0$ has the same
distribution as $\wt{X}_0$, then the process $(X_t)_{t\geq 0}$ has
the same law as the process $(\wt{X}_t)_{t\geq 0}.$ We say pathwise
uniqueness holds for (\ref{mainsde}) if whenever $(X,L)$ and
$(\wt{X},L)$ are two weak solutions of (\ref{mainsde}) with
common $L$ on a common probability space (w.r.t. possibly different filtrations) and with the same initial condition, then $\P(X_t = \wt{X}_t \mbox{ for all } t \geq 0) = 1.$ We
say that strong uniqueness holds for (\ref{mainsde}) if whenever $X$
and $\wt{X}$ are two strong solutions of (\ref{mainsde}) relative
to $L$ with the common initial condition $X_0$, then $\P(X_t = \wt{X}_t
\mbox{ for all } t \geq 0) = 1.$ Clearly, pathwise uniqueness implies
strong uniqueness.

Note that for $\beta>0$ Definition~\ref{D:sol} coincides with the
standard definition of a solution. 

\medskip

We are now ready to present our main result.
\begin{Theorem}\label{T:1}
For any $x\in\R$, $b \in \C^{\beta}$, $\alpha\in(1,2)$, $\beta > \frac{1-\alpha}{2}$, stochastic differential equation \eqref{mainsde} has a unique strong solution.
\end{Theorem}

Let us give a few comments about the above result. As mentioned earlier, the
case of $\beta>0$ was resolved in Tanaka, Tsuchiya,
Watanabe~\cite{bib:ttw74}. Thus the main result of this paper is
for the case $\beta\in ( \frac{1-\alpha}{2}, 0]$. Though we do not
 consider the case $\alpha =2$ (that is,  when $L$ is replaced by the standard
 Brownian motion $B$), our proof can be suitably modified to show that Theorem \ref{T:1}
 holds for case $\alpha =2$ ( with an appropriate replacement of $L$ by $B$ in Definition~\ref{D:sol}). We note that in this case the result  would be less restrictive than the corresponding result in \cite{BC}. Indeed, we 
 allow $(b_n)_{n\in\Z_+}$ to be an arbitrary sequence of smooth functions approximating $b$, whereas  \cite{BC} imposes  extra conditions on regularity of $A^n$ (cf. Definition~\ref{D:sol} and \cite[condition (iii) of Definition~2.1]{BC}).
 
  Note also that if $\beta<0$, then the upper bound for the H\"older exponent in condition~3 of the definition of the solution is less than $1$. Thus the process $A$ might be of infinite variation and $X$ might not be a semimartingale. On the other hand, if $\beta>1/2-\alpha/2$, then $1+\beta/\alpha>1/2$ and thus condition \eqref{3cond} holds with $\kappa>1/2$. This implies that the quadratic variation of $A$ is $0$.

\begin{Remark} We note that in our setting condition 2  of  Definition~\ref{D:sol} follows from a weaker condition.
 Let $ X=x+ A+L$ be a c\`adl\`ag adapted process that satisfies
 conditions 1 and 3 of Definition~\ref{D:sol} and assume
 that (2.1) holds only for
 {\em a particular} sequence of smooth functions $(\wt b_n)_{n\in\Z_+}$
 converging to $b$ in $\C^\beta$. Then, under the hypothesis of
 Theorem~\ref{T:1}, it follows from our proofs that    
\eqref{sxodka} holds also for  {\em any other} smooth sequence $(b_n)_{n\in\Z_+}$ converging to $b$ in $\C^\beta$.
\end{Remark}

The proof of {Theorem}~\ref{T:1} consists of two parts: namely, existence and
uniqueness. Usually proving existence is an ``easy'' part of this type of theorems. Indeed, in
the case when the coefficients in the stochastic equation are sufficiently
regular, it is possible to directly show strong existence. For
example, for equation~\eqref{mainsde} when $b$ is a bounded continuous
function strong existence follows via a simple compactness argument (see the 
comment before \cite[Lemma~4.1]{Pr12}). However, for the equations with a generalized drift the situation is much more complicated, since even the notion of a solution should be defined very carefully.

In the intermediate steps of our proof  we will use additionally the notion of a {\it virtual solution}, which has been introduced recently for related equations with distributional drift
(see \cite[Definition 25]{bib:fir17}). The notion is based on applying
a Zvonkin type transformation to the equation of interest and
obtaining a ``transformed'' equation where the drift is more regular. 
The broad strategy of the proof then involves
showing existence and uniqueness for the ``transformed'' equation;
its solution is called a ``virtual solution''. However it is not obvious at all how to identify the concept
of a solution to (1.2) with the virtual solution. Recently it was shown in~\cite{bib:zz17} for some
multidimensional equations driven by the Brownian motions that
the virtual solutions are solutions to the martingale problem associated with
the original equation.

It is technically challenging to carry out the above program for
proving Theorem~\ref{T:1}; in particular, as mentioned before, the solution
$X=(X_t)_{t \geq 0}$ will not be a semimartingale. So the classical tools
and methods will not be applicable. The novelty of our approach is in working
with the correct notion of a (natural) solution, a suitable adaptation
of the transformation, along with a specific technique for the identification of solutions and virtual solutions. 
We provide more details in the next subsection.

\subsection{Overview of the proof of {Theorem}~\ref{T:1}} \label{sec:overview}

The proof of Theorem~\ref{T:1} consists of a number of steps. To ease the comprehension of the
proof and for the convenience of the reader, we provide the following road map of the proof.

\begin{Assumption}
For the rest of the paper we fix $\alpha\in(1,2)$, $\beta\in(\frac{1-\alpha}{2},0)$, $b\in\C^{\beta}$, the initial
condition $x \in \R$ and the length of the time interval $T>0$.
\end{Assumption}
Our goal is to show that for the parameters chosen above on the time
interval $[0,T]$ stochastic differential equation \eqref{mainsde}
has a unique strong solution. Note also that we do not lose the
generality by choosing $\beta<0$, since if we prove existence and
uniqueness for any $b\in \C^{\beta_0}$, then it also holds for any
$b\in \C^{\beta}$ with $\beta\ge \beta_0$.

As mentioned earlier in the introduction, to study \eqref{mainsde} we use a new version of the drift--transformation method; the original method dates back to Zvonkin \cite{zvonkin74} and Veretennikov \cite{ver80}. Heuristically, the main idea of the method can be formulated as follows: SDE \eqref{mainsde} has a ``very bad'' drift (recall that $b$ is not even a function, but just a distribution) but relatively ``good'' diffusion. Therefore one can make the following trade-off. With the help of a certain auxiliary function $u\colon\R_+\times\R\to\R$, one can consider the process $Y_t:=u(t,X_t)$, $t\ge0$. This process satisfies a new SDE (which we will call the Zvonkin equation) with better drift and worse (though still not ``too bad'') diffusion. If one can prove that this Zvonkin SDE has a unique strong solution and the function $u$ is ``nice'', then this would imply that the original SDE \eqref{mainsde} also has a unique strong solution.

Note that in our case implementing the above algorithm is very tricky.
It is rather hard to show directly the strong existence of the solutions (even)
to the Zvonkin equation. Thus, we take the following route. First we derive the Zvonkin equation and construct a weak solution to it.
 Then we use this solution to construct a weak solution to \eqref{mainsde}. After it we show that
strong uniqueness holds for the Zvonkin equation (and hence for \eqref{mainsde}). Finally we apply a
generalized version of the classical Yamada-Watanabe theorem (see
\cite{K14}) to establish strong existence for \eqref{mainsde}.

Thus, everything depends on the choice of the transformation function
$u$. In the original papers \cite{zvonkin74} and \cite{ver80} the
function $u$ was a solution of a certain partial differential
equation. Motivated by \cite{FGP}, Priola \cite{Pr12} suggested to
take a different $u$, which arises from a family of resolvent
equations. We further develop Priola's approach to accommodate distributional drift.

To present the equation on $u$ we need to recall a couple of notions.

\begin{Definition} \label{gensemgp} Let $\{P_t\}_{t \geq 0}$ be the Markov semigroup associated with the symmetric one-dimensional $\alpha$-stable process $L$. Let $\OL_\alpha$, be the infinitesimal generator of $\{P_t\}_{t \geq 0}$ with domain ${\mathcal D}(\OL_\alpha)$. 
\end{Definition}
It is well known that $\OL_\alpha$ is the fractional Laplace operator $-(-\Delta)^{\alpha/2}$. Further, $\S\subset {\mathcal D}(\OL_\alpha)$ and 
if $f\in \S$, then
$$\OL_\alpha f(x)= \int_\R \bigl(f(x+y)-f(x)-yf'(x)\I_{|y|\le1}\bigr) |y|^{-1-\alpha}\, dy.$$
We extend the definition of $\OL_\alpha$ to the space of all Schwarz distributions in the standard way. Namely,
for $f\in\S'$ we set 
$$ \langle \OL_\alpha f,\phi\rangle:=\langle f,\OL_\alpha\phi\rangle,\quad\phi\in\S. $$
 
We will be also dealing with products of a function and a distribution. In this regard, let us
recall that if $f\in\C^{\gamma_1}$ and $g\in\C^{\gamma_2}$, where $\gamma_1,\gamma_2\in\R$ and $\gamma_1+\gamma_2>0$, then the product $fg$ is well defined as a distribution. More
precisely, the map $(f,g)\to f g $ extends to a continuous bilinear map from $\C^{\gamma_1}\times\C^{\gamma_2}\to \C^{\gamma_1\wedge\gamma_2}$, see, e.g., \cite[Corollary 1]{GP15}.

Now we can present the equation on the transformation function $u$. We consider the following equation
\begin{equation}\label{me}
\lambda u -\OL_{\alpha} u - f u'= g,
\end{equation}
where $\lambda>0$, $f,g\in\C^\eta$, $\eta \in \R$. We understand this equation in the distributional sense: we say that $u \in \C^\gamma$ is a solution to \eqref{me} if $\gamma>1-\eta$ and for any $\phi\in\S$
\begin{equation*}
{\langle \lambda u -\OL_{\alpha} u - f u^\prime, \phi \rangle = \langle g,\phi\rangle.}
\end{equation*}
We note that the term $fu^\prime$ above involves the product of a
function and a distribution. However, thanks to our additional
assumption $\gamma>1-\eta$ and the explanations above, this product is well--defined.

Clearly, for $\eta >0$ and $\gamma > \alpha$  equation \eqref{me} can be interpreted pointwise.

 We will call \eqref{me} \textit{the
resolvent equation} and we are going to use it extensively throughout
the proof. In different parts of the proof we will be substituting $f$
and $g$ by the drift $b$, its smooth approximations $b_n$ or sometimes
just by $0$. For brevity, we will say $u^{\lambda}_{f,g}$ solves
\eqref{me} to imply that $u^{\lambda}_{f,g}$ is a solution to
\eqref{me} with the parameters $\lambda$, $f$ and $g$. However, if it
is clear from the context, we may drop the additional indices.

Note that the difference of our approach and \cite{Pr12} is
that we allow $f$ and $g$ in \eqref{me} to be distributions (and not
just regular functions). It will make establishing corresponding
estimates much more trickier; on the other hand it will allow us to
deal with the distributional drift in our main SDE
\eqref{mainsde}.

 Our first step is to show that the resolvent
 equation \eqref{me} is actually well--defined. That is, it has a
 unique solution with prescribed regularity and possesses a continuity
 property.

\begin{Proposition} \label{p:resolvent}
For any $\eta> \frac12-\frac\alpha2$ and $M >0$ there exists $\lambda_0=\lambda_0(\eta,M)$ such that for any $\lambda\ge\lambda_0$ and any $f,g\in\C^\eta$ with $\|f\|_\eta \le M$
the following holds:
 \begin{enumerate}
 \item[{\rm(i)}] there exists a unique solution, $u^\lambda_{f,g}$ to \eqref{me} in class $\C^{\frac{1+\alpha}{2}}$. Furthermore, for each $\gamma\in[0\vee\eta,\alpha+\eta)$ we have  $u^{\lambda}_{f,g}\in\C^{\gamma}$ and there exists a constant $C=C(\eta,\gamma)>0$ such that
\begin{equation}\label{est}
\|u^{\lambda}_{f,g} \|_{\gamma} \le C \,\lambda^{-1-\frac\eta\alpha+\frac\gamma\alpha} \,\|g\|_\eta\,(1+ \|f\|_\eta);
\end{equation}

\item[{\rm(ii)}] for any sequences of functions $(f_n)_{n\in\Z_+}$, $(g_n)_{n\in\Z_+}$ such that
$\|f_n\|_{\eta}\le M$, $n\in\Z_+$ and
$\|f_n-f\|_{\eta}\to0$, $\|g_n-g\|_{\eta}\to0$ as $n\to\infty$ we have
\begin{equation*}
\|u^\lambda_{f_n, g_n}-u^\lambda_{f,g}\|_{(1+\alpha)/2}\to0,\,\,\text{ as $n\to\infty$}.
\end{equation*}
\end{enumerate}
\end{Proposition}

Our next step is to derive the Zvonkin equation, which is more
challenging in our case due to the fact that $b$ is a
distribution. Let $u^{\lambda}_{b}= u^{\lambda}_{b,b}$ be the unique
solution to \eqref{me}, which exists by Proposition~\ref{p:resolvent}.
We would like to apply the Zvonkin--type transform $\phi(x) = x +
u^\lambda_{b}(x)$ to { $X$ solving} \eqref{mainsde}. As mentioned
earlier, a solution to SDE \eqref{mainsde} is not a
semimartingale. Thus we cannot use the standard It\^{o}
formula and have to employ and develop
the theory of Dirichlet processes. Eventually, we derive the following result.

\begin{Proposition} \label{p:vtreal}
Let $\lambda_0=\lambda_0(\beta,2\| b \|_\beta)$ be as in Proposition~\ref{p:resolvent} and l $\lambda \ge\lambda_0$. Let $X=(X_t)_{t\in[0,T]}$ be a weak solution to \eqref{mainsde} and $u^{\lambda}_{b}=u^{\lambda}_{b,b}$ be the unique solution to \eqref{me}. Then for any $t\in[0,T]$
\begin{equation} \label{usde}
 u^\lambda_{b}(X_t)+ X_t= u^\lambda_{b}(x)+x+ \lambda \int_0^t u^\lambda_{b}(X_s) ds + \int_0^t\int_\R [ u^\lambda_{b}(X_{s-} + r) - u^\lambda_{b}(X_{s-})]  \wt N(ds, dr)+L_t.
\end{equation}
\end{Proposition}
As mentioned before, we will call the SDE \eqref{usde} the Zvonkin equation. Note that: first, all the terms in the Zvonkin equation make sense; second, \eqref{usde} does not have any distributional drift term like \eqref{mainsde}; and finally, if $\lambda$ is very large but still finite, then $u^\lambda_{b}$ is very close to zero and thus the only term with $X_t$ that will not disappear in \eqref{usde} (apart from $X_t$ itself) is $\lambda \int_0^t u^\lambda_{b}(X_s) ds$, which is smooth in $t$ and behaves ``nicely''.

To show existence of a weak solution for (\ref{mainsde}) and \eqref{usde}, we
construct an approximating sequence. Let $(b_n)_{n\in\Z_+}$ be a sequence of functions in $\C^\infty_b$ such that $b_n$
converges to $b$ in $\C^\beta$ and $\|b_n\|_\beta\le 2 \|b\|_\beta$.
Let $X^n=(X_t^n)_{t\in[0,T]}$, $n\in\Z_+$ be the strong solution to the
following stochastic differential equation:
\begin{equation}\label{appr}
X^n_t=x+ \int_0^tb_n(X^n_s)\,ds + L_t,\quad t\in[0,T]
\end{equation}
and put $A^n=(A_t^n)_{t\in[0,T]}$, $n\in\Z_+$
\begin{equation}\label{Ant}
A^n_t=\int_0^tb_n(X^n_s)\,ds,\quad t\in[0,T].
\end{equation}
The strong existence and uniqueness for (\ref{appr}) is well known
(see for e.g., \cite[Theorem~6.2.3]{Apple}). To show tightness and
subsequential limit of the above sequence, we will use the Zvonkin
transformation. We obtain the following result. Let Skorokhod space
$\D_\R[0,T]$ be the space of all c\`adl\`ag functions from $[0,T]$ to $\R$.

\begin{Proposition}\label{p:wex} Let $\lambda_0=\lambda_0(\beta,2\| b \|_\beta)$ be as in Proposition~\ref{p:resolvent} and $\lambda \ge\lambda_0$. Let  $u^{\lambda}_{b}=u^{\lambda}_{b,b}$ be the unique solution to \eqref{me} and let $(X^n,A^n)$ be defined as above. Then there exists a subsequence $n_k$ such that $(X^{n_k},A^{n_k})$ converges weakly to $(X,A)$ in $\D_\R[0,T]$. Further,
 \begin{enumerate}
 \item[{\rm(i)}] $X$ is a weak solution of the Zvonkin equation \eqref{usde}.
  \item[{\rm(ii)}] $X$ is a weak solution to stochastic differential equation \eqref{mainsde}.
\end{enumerate}
\end{Proposition}

Our final ingredient is to establish pathwise uniqueness for
(\ref{usde}).
\begin{Proposition}\label{p:pu}
There exists $\lambda_1=\lambda_1(\beta,\| b \|_\beta)$ such that for any $\lambda >\lambda_1$ the Zvonkin equation \eqref{usde} has a pathwise unique solution.
\end{Proposition}

We now have all the key ingredients to complete the proof of Theorem
\ref{T:1}. Proposition \ref{p:wex}(ii) shows that (\ref{mainsde}) has
a weak solution. Proposition \ref{p:pu} and Proposition \ref{p:vtreal} together show
that (\ref{mainsde}) has pathwise uniqueness. We then use Theorem 3.4
in \cite{K14}, to establish the classical Yamada-Watanabe theorem, which
implies existence and uniqueness of a strong solution in our general
setting. We present the details in Section \ref{sec:kur}.

\smallskip
The rest of the paper is organized as follows. In Section~\ref{S:prelimka} we present a number of preliminary results that are used for the proof of the theorem. Most of them are very well-known and are provided for the sake of completeness. We discuss the basic properties of the Besov norms, convergence in the Skorokhod space, and Dirichlet processes. Section~\ref{S:2} is devoted to the proof of
Proposition~\ref{p:resolvent}. In Section~\ref{sec:5} we give the proofs of
Propositions~\ref{p:vtreal}, \ref{p:wex}, \ref{p:pu}. This allows to finish the proof of~Theorem \ref{T:1} in Section
 \ref{sec:kur}. Some auxiliary results are proved in the Appendix.

\smallskip
\noindent \textbf{Convention on constants}. Throughout the paper $C$
 denotes a positive constant whose value may change from line to line.
 All other constants will be denoted by $C_1,C_2,\ldots$ They are all
 positive and their precise values are not important. The dependence
 of constants on parameters if needed will be mentioned inside
 brackets, e.g, $C(\alpha, \beta)$.

\section{Preliminaries}\label{S:prelimka}

\subsection{Besov norms and Fractional Laplacian} \label{S:11}
In this section we collect some standard properties of Besov norms and fractional Laplacian that will be used throughout the paper.

First we recall that for $\gamma \in (-\infty,0) \cup \{ (0,\infty) \setminus \N \}$ there
exists the following equivalent definition of Besov-H\"older spaces. Below the notation $\|f\|_\gamma\simeq a$ means that $\|f\|_\gamma= C a$ for some universal constant $C>0$ that depends
only on $\gamma$ but not on the function $f$. 

\begin{Lemma}\label{Besov}
For $\gamma=n+\rho$, $n\in\Z_+$, $\rho\in(0,1)$, the space $\C^\gamma$ is the space of bounded and $n$ times differentiable functions whose $n$th derivative is $\rho$-H\"older continuous. Further, for $f\in\C^\gamma$
\[ \|f\|_\gamma \simeq \sum_{k=0}^n \snorm{\frac{d^k f}{dx^k}} + \sup_{x \neq y } \frac{\mid \frac{d^nf}{dx^n}(x) - \frac{d^nf}{dx^n}(y) \mid}{\mid x - y \mid^\rho}. \]

For $\gamma <0$ we have $\C^\gamma\subset S'$. Further, for $f\in\C^\gamma$
\begin{equation*}
   \| f \|_\gamma \simeq \sup \left \{\lambda^{-\gamma}|\langle f, \phi_{\lambda,x}\rangle|: \begin{array}{l} x \in \R, \lambda \in (0,1], \phi_{\lambda,x} (\cdot) := \frac{1}{\lambda}\phi(\frac{\cdot \ - x}{\lambda})  \mbox{ with } \phi \in \S,\\ \supp \phi \subset (-1,1), \sup\limits_{n \leq \lceil{-\gamma}\rceil } \snorm{\frac{d^n\phi}{dx^n}} \leq 1. \end{array}\right \}.
\end{equation*}

\end{Lemma}

\begin{proof} This property is standard and it follows from, e.g., \cite[Exercise 13.31]{FH}.
 \end{proof}

Note also that for any $\gamma\in\R$ the space $\C^\gamma$ includes all distributions which are distributional derivatives of elements of $\C^{\gamma+1}$.

The next lemma provides useful properties of Besov norms.
\begin{Lemma}\label{L:31} Let $f$ be a function $\R\to\R$. Then the following holds:
 \begin{enumerate}
  \item[\rm{(i)}] For any $\eta, \gamma \in \R$ and $\eta < \gamma$, we have $\|f\|_{\eta}\le \| f\|_{\gamma}$.
 \item[\rm{(ii)}] For any $\eta, \gamma \in \R$ and $\eta<0<\gamma$, there exist constants $C_1>0$, $C_2>0$ such that $$\|f\|_{\eta}\le C_1 \snorm{f}\le C_2 \|f\|_{\gamma}.$$
 \item[\rm{(iii)}] For any $\eta \in\R$ there exists $C>0$ such that
 \begin{equation}\label{der}
\|f'\|_{\eta-1}\le C \|f\|_{\eta}.
 \end{equation}
 \item[\rm{(iv)}] For any $\eta, \gamma \in \R$, $\eta + \gamma >0$ there exists $C>0$ such that for any $g \in \C^\gamma$,
 \begin{equation}\label{product}
 \|fg\|_{\eta \wedge \gamma}\le C\|f\|_{\eta}\|g\|_{\gamma}.
 \end{equation}
\end{enumerate}
Further, the constants $C$, $C_1$, $C_2$ do not depend on the functions $f$, $g$.
\end{Lemma}
\begin{proof} (i) and (ii) follow, e.g., from Exercise 2 in \cite{P15}. (iii) This follows from, e.g., \cite[Formula 2.3.8.(6)]{T}.
(iv) Follows immediately from, e.g., \cite[Section~2.3]{P15} and \cite[Theorem 13]{P15}.
 \end{proof}
The properties of Besov norms established in Lemma~\ref{L:31} are basic and we will be using them further in the paper without an explicit reference to the lemma. Our last lemma in this section describes additional properties of Besov norms in relation to the fractional Laplacian and its associated semigroup.

\begin{Lemma}\label{P:1}
 \begin{enumerate}
  \item[\rm{(i)}] For any $\gamma\in\R$ there exists $C>0$ such that for any $f\in\C^\gamma$
\begin{equation*}
\|\OL_\alpha f\|_{\gamma-\alpha}\le C \|f\|_{\gamma}.
\end{equation*}
\item[\rm{(ii)}]For any $\gamma\ge0$, $\eta\in(-\infty,\gamma]$ there exists $C>0$ such that for any $f\in\C^\infty_b$, $t\in(0,1]$
\begin{equation*}
\|P_t f\|_{\gamma}\le C t^{\frac{\eta-\gamma}{\alpha}}\|f\|_{\eta}
\end{equation*}
\item[\rm{(iii)}]For any $\gamma\ge0$, $\eta\in(-\infty,\gamma]$ there exists $C>0$ such that for any $f\in\C^\infty_b$, $t\ge1$
\begin{equation*}
\|P_t f\|_{\gamma}\le C\|f\|_{\eta}.
\end{equation*}
\end{enumerate}
\end{Lemma}
\begin{proof} (i), (iii). These statements are standard however we were not able to find their proofs in the literature; for the sake of completeness we provide their proofs in
Appendix~\ref{p32a}. (ii) We refer the reader to \cite[Lemma A.7]{GIP17}. Even though the statement of Lemma A.7 is for a bounded set, one can verify that the proof works also for $\R$.
 \end{proof}

\subsection{Properties of convergence in the Skorokhod space}

Let $E$ be a metric space. In this section we provide some technical though important tools for studying convergence in the Skorokhod space $\D_E[0,T]$, that is the space of all c\`adl\`ag functions $[0,T]\to E$. We refer the reader to \cite[Chapter~3]{Bilya} and \cite[Chapter~3]{EK86} for a detailed treatment of the Skorokhod space and the necessary definitions. Denote by $d$ the Skorokhod distance in $\D_E$. Let $\Lambda$ be the set of continuous strictly increasing functions mapping $[0,T]$ onto $[0,T]$.

We will use the following lemma, which is just a minor modification of a corresponding lemma from \cite{KP91}. For the convenience of the reader and for the sake of exposition we state this proposition and its proof below.

\begin{Lemma}[cf. {\cite[Lemma~2.1]{KP91}}]\label{P:PK}
Let $E_1$ and $E_2$ be metric spaces and let $\Phi$, $\Phi_n$, where $n\in\Z_+$, be mappings
$\D_{E_1}[0,T]\to\D_{E_2}[0,T]$, $T>0$. Suppose
that for any $n\in\Z_+$ we have $\Phi^n(Z\circ\mu)=\Phi^n(Z)\circ\mu$ whenever $Z\in\D_{E_1}[0,T]$, $\mu\in\Lambda$. Further assume that $Z^n\to Z$ in the uniform metric implies $\Phi^n(Z^n)\to\Phi(Z)$ in the uniform metric. Then $Z^n\to Z$ in the Skorokhod topology implies $(Z^n,\Phi^n(Z^n))\to(Z,\Phi(Z))$ in the Skorokhod topology.
\end{Lemma}
\begin{proof}
The proof follows the proof of \cite[Lemma~2.1]{KP91} with minor modifications. Let $\rho_{E_1}$ and $\rho_{E_2}$ denote the metric on $E_1$ and $E_2$ respectively. Let $\rho_{E_1\times E_2}$ be the product metric on $E_1\times E_2$.

Take any $Z\in \D_{E_1}[0,T]$, a sequence $(Z^n)_{n\in\Z_+}$, $Z^n\in \D_{E_1}[0,T]$ and assume that $d(Z^n,Z)\to0$. Then there exists a mapping $\mu^n\in\Lambda$ such that $\snorm{\mu^n-I}\to0$ where $I$ denote the identity map and $\sup_{t \in [0,T]}\rho_{E_1}( (Z^n\circ \mu^n)(t), Z_t)\to0 $.  By the assumptions, this implies that
$$
\sup_{t \in [0,T]} \rho_{E_2}\bigl( \Phi^n(Z^n\circ\mu^n)(t),\Phi(Z)(t)\bigr) \to0.
$$
Since $\Phi^n(Z^n\circ\mu^n)=\Phi^n(Z^n)\circ\mu^n$, we finally obtain
\begin{align*}
d\bigl((Z^n,&\Phi^n(Z^n)),(Z,\Phi(Z))\bigr)\\
&\le \snorm{\mu^n-I}+ \sup_{t \in [0,T]} \rho_{E_1 \times E_2}\Bigl(\bigl([Z^n\circ\mu^n](t),[\Phi^n(Z^n)\circ\mu^n](t)\bigr),\bigl(Z_t,\Phi(Z)(t)\bigr)\Bigr)\to0.
\end{align*}
This implies the statement of the lemma.
\end{proof}

We will use the following simple lemma that deals with convergence of integrals in the Skorokhod space.

\begin{Lemma}\label{L:425}
Let $(X^n)_{n\in\Z_+}$ be a sequence of elements in $\D_{\R}[0,T]$ converging a.s. in the Skorokhod metric to
$X$.
Let $(f_n)_{n\in\Z_+}$ be a sequence of continuous functions $\R^2\to\R$ converging uniformly to $f$. Assume
\begin{equation*}
\sup_{n\in\Z_+}\snorm{f_n}<\infty.
\end{equation*}
Let $A\in\R$ be a Borel measurable set and let $\theta$ be a finite measure on $A$.

Then the process $J^n$ defined as
\begin{equation*}
J^n(t):=\int_0^t\int_A f_n(X^n_{s-},r)\,\theta(dr)ds,\quad t\in[0,T],
\end{equation*}
uniformly on $[0,T]$ converges to the process $J$
\begin{equation*}
J(t):=\int_0^t\int_A f(X_{s-},r)\,\theta(dr)ds,\quad t\in[0,T]
\end{equation*}
a.s. as $n\to\infty$.
\end{Lemma}
\begin{proof}
We have
\begin{equation*}
\snorm{J^n-J}\le \int_0^T \int_A |f_n(X^n_{s-},r)-f(X_{s-},r)|\,\theta(dr)ds.
\end{equation*}
Recall that $X^n$ converges to $X$ a.s. in the Skorokhod metric. Therefore for almost all $\omega\in\Omega$ the sequence  $(X_{t-}^n(\omega))_{n\in\Z_+}$ converges to $X_{t-}(\omega)$ for all but countably many $t\in[0,T]$. Since the sequence $(f_n)_{n\in\Z_+}$ is uniformly bounded and converges pointwise to $f$, we can apply the dominated convergence theorem to conclude
\begin{equation*}
\snorm{J^n-J}\to0\text{\, a.s. as $n\to\infty$}.\qedhere
\end{equation*}
\end{proof}

The next lemma is more complicated and it deals with passing to the limit in stochastic integrals.

\begin{Lemma}\label{L:42}
Let $(X^n,L^n)_{n\in\Z_+}$ be a sequence of elements in $\D_{\R^2}[0,T]$ converging a.s. in the Skorokhod metric to $(X,L)$. Let $\nu$ be a Levy measure satisfying
\begin{equation}\label{rkvadrat}
\int_{\R} (1\wedge r^2) \nu(dr)<\infty.
\end{equation}
Assume that $L$, $(L^n)_{n\in\Z_+}$ are L\'evy processes with the L\'evy measure $\nu$. Let $N^n$ and $\wt N^n$ be the Poisson measure and compensated Poisson measure associated with $L^n$, respectively, where $n\in\Z_+$. Define $N$ and $\wt N$ in a similar way.
Let $(\mathcal{F}^{X^n}_t)_{t\ge 0}$ (respectively $(\mathcal{F}^{X}_t)_{t\ge 0}$) be the natural filtration of $X^n$ (respectively $X$). Assume that $\{\wt N^n_t(A), t\geq 0\}$ is an $\mathcal{F}^{X^n}_t$-martingale, for any compact set $A\subset \R\setminus \{0\}$. Then
\begin{itemize}
\item[{\rm(i)}]
$\{\wt N_t(A), t\geq 0\}$ is an $\mathcal{F}^{X}_t$-martingale, for any compact set
 $A\subset \R\setminus \{0\}$.
\item[{\rm(ii)}]
Let $(f_n)_{n\in\Z_+}$ be a sequence of continuous functions $\R^2\to\R$ converging uniformly to $f$. Assume that for some $C>0$
\begin{equation}\label{cond1}
|f_n(x,r)|\le C (|r|\wedge1),\quad x,r\in\R,\,\,n\in\Z_+.
\end{equation}
Then for any $T>0$ the process $I^n$ defined as
\begin{equation*}
I^n(t):=\int_0^t \int_\R f_n(X^n_{s-},r)\wt N^n(ds,dr),\quad t\in[0,T]
\end{equation*}
converges in probability in the Skorokhod space $\D_{\R}[0,T]$ to the process $I$
\begin{equation*}
I(t):=\int_0^t \int_\R f(X_{s-},r)\wt N(ds,dr),\quad t\in[0,T].
\end{equation*}
\end{itemize}
\end{Lemma}

The proof of this Lemma is provided in the Appendix~\ref{A:L42}.

\subsection{Dirichlet processes}\label{S:Dirya}

Recall that a weak or strong solution to SDE \eqref{mainsde} is not
necessarily a semimartingale; it belongs to a more general class of
processes called Dirichlet processes. Thus to study properties of
solutions of our equation we need to develop some parts of
theory of Dirichlet processes. This is done in this section. We begin
with the following definitions.

\begin{Definition}[{\cite{Foll}}] We say that a continuous adapted process $(A_t)_{t\in[0,T]}$ is a \textit{process of zero energy} if $A_0=0$ and
\begin{equation*}
\lim_{\delta \rightarrow 0} \sup_{\pi_{T} \, : \,
   \mid \pi_{T} \mid < \delta} \E \left( \sum_{t_i\in\pi_T} \mid A_{t_{i+1}} - A_{t_{i}} \mid^2 \right )= 0
  \end{equation*}
where $\pi_{T}$ denotes a finite partition of $[0,T]$ and $|\pi_{T}|$ denotes the mesh size of the partition.
\end{Definition}

\begin{Definition}[{\cite{Foll}}] \label{D:decompos} We say that an adapted process $(X_t)_{t\in[0,T]}$ is a \textit{Dirichlet process} if
\begin{equation}\label{decompos}
X_t=M_t+A_t,\quad t\in[0,T],
\end{equation}
where $M$ is a square--integrable martingale and $A$ is an adapted process of zero energy.
\end{Definition}

It was proven in \cite{Foll} that such decomposition \eqref{decompos}
of a Dirichlet process $X$ is unique. Thus, we see that the class of
Dirichlet processes naturally extends the class of
semimartingales. Note however that since the process $A$ in
decomposition \eqref{decompos} might be of infinite variation, the integral
with respect to $A$ might be not well-defined in the classical
sense. The next definition extends the notion of a stochastic integral
to the class of integrals with respect to zero-energy
processes. We shall define the integral with respect to the
 zero--energy process $A$ as a limit in probability of the
 corresponding forward Riemann sums.
\begin{Definition}[{\cite[page~90]{CJMS}}]\label{D:Dirint}
Let $f\colon\R\to\R$ be a bounded continuous function with a bounded continuous derivative. Let $X$ be a Dirichlet process with decomposition \eqref{decompos}. For ${n\in\Z_+}$ let $D_n := \{t^n_i\} $ be a sequence of \textit{refining} (and non-random) partitions of $[0,T]$ whose mesh size tends to $0$ as $n\to\infty$. Then
\begin{equation*}
\int_s^t f(X_r) d A_r :=(\P) \lim_{n\to\infty} \sum_{t_i^n\in D_n,\,t_i^n\in[s,t)}
f(X_{t_i^n})(A_{t_{i+1}^n}-A_{t_i^n}),\quad 0\le s\le t \le T
\end{equation*}
(if exists), where the limit is taken in probability .
\end{Definition}

It was shown in \cite[Theorem~3.1]{CJMS} that the integral $\int_0^\cdot f(X_s) d A_s$ exists and does not depend on the choice of the sequence of refining partitions $D_n$.

We will need a couple of statements describing further properties of integrals with respect to the Dirichlet processes. Some of these results (and their proofs) are close in spirit to \cite[Lemma~2.3]{BC}.

In the first lemma we prove that under certain regularity conditions, the convergence in probability in the definition of the Dirichlet integral can be improved to convergence in $\SL_p$.
\begin{Lemma}\label{L:mainyoung}
Let $(X,A)$ be as in Definition~\ref{D:decompos}.
Let $f\colon\R\to\R$ be a bounded continuous function with a bounded continuous derivative. Suppose that for some $p_1, p_2>0$ and $\gamma_1, \gamma_2>0$ with $\gamma_1+\gamma_2>1$ and $1/h:=1/p_1+1/p_2\le 1$ there exist constants $C_f, C_A\ge1$ such that for any $s,t\in[0,T]$
\begin{align}
\E|f(X_t)-f(X_s)|^{p_1}&\le \left(C_{f}\right)^{p_1}|t-s|^{p_1\gamma_1},\,\,\E|f(X_t)|^{p_1}\le \left(C_{f}\right)^{p_1}, \label{condej1}\\
\E|A_t-A_s|^{p_2}&\le \left(C_{A}\right)^{p_2}|t-s|^{p_2\gamma_2} \label{condej2}.
\end{align}

\begin{itemize}
\item[{\rm(i)}] Then for any $0\le s\le t\le T$ the sequence of partial sums
\begin{equation*}
I_n:= \sum_{i=0}^{2^n-1} f(X_{t^i_n})(A_{t^{i+1}_n}-A_{t^i_n}),
\end{equation*}
where
 \begin{equation*}
 t^k_n:=s+k2^{-n}(t-s),\quad\text{for }k=0,1,\hdots,2^n;\,\, n\in\Z_+,
\end{equation*}
converges to $I:=\int_s^t f(X_r) dA_r$ in $\SL_h$.
\item[{\rm(ii)}] Moreover, there exists $C=C(T,\gamma_1,\gamma_2)>0$ such that for any $0\le s\le t\le T$, $n\in\Z_+$, we have the following estimate of the remainder term:
\begin{equation}\label{Lhnormdirrem}
\|I-I_n\|_{\SL_h}\le CC_{f}C_{A}2^{-n(\gamma_1+\gamma_2-1)}.
\end{equation}
\item[{\rm(iii)}] Finally, there exists $C=C(T,\gamma_1,\gamma_2)>0$ such that for any $0\le s\le t\le T$
\begin{equation}\label{Lhnormdir}
\|\int_s^t f(X_r) dA_r\|_{\SL_h}\le CC_{f}C_{A}(t-s)^{\gamma_2}.
\end{equation}
\end{itemize}
\end{Lemma}
\begin{proof}
(i) Fix $0\le s<t\le T$. For $n\in\Z_+$ put $D_n:=\{t^0_n, t^1_n,\dots,t^{2^n}_n\}$. Clearly $(D_n)_{n\in\Z_+}$ is a sequence of refining partitions of $[s,t]$. Therefore, it follows from Definition~\ref{D:Dirint} that $I_n$ converges to $I$ in probability as $n\to\infty$.

 On the other hand, it is easy to check that for any $n\in\Z_+$
 \begin{equation*}
I_{n+1}-I_n=\sum_{i=0}^{2^n-1} (f(X_{t^{2i+1}_{n+1}}-f(X_{t^{i}_{n}}))(A_{t^{i+1}_n}-A_{t^{2i+1}_{n+1}}).
 \end{equation*}
 Therefore, applying successively the Minkowski and H\"older inequalities, we get
 \begin{align}\label{a-la-Khoa}
\|I_{n+1}-I_n\|_{\SL_h}\le&\sum_{i=0}^{2^n-1} \|(f(X_{t^{2i+1}_{n+1}}-f(X_{t^{i}_{n}}))(A_{t^{i+1}_n}-A_{t^{2i+1}_{n+1}})\|_{\SL_h}\nn\\
\le&\sum_{i=0}^{2^n-1} \|f(X_{t^{2i+1}_{n+1}}-f(X_{t^{i}_{n}})\|_{\SL_{p_1}}\|(A_{t^{i+1}_n}-A_{t^{2i+1}_{n+1}})\|_{\SL_{p_2}}\nn\\
\le& C_{f}C_{A} 2^n 2^{-n(\gamma_1+\gamma_2)}(t-s)^{\gamma_1+\gamma_2},
 \end{align}
 where we also used conditions \eqref{condej1} and \eqref{condej2} and the fact that $1/h=1/p_1+1/p_2$. Similarly,
 \begin{equation}\label{inol}
\|I_{0}\|_{\SL_h}\le \|f(X_{s})(A_{t}-A_{s})\|_{\SL_h}\le \|f(X_{s})\|_{\SL_{p_1}}\|A_{t}-A_{s}\|_{\SL_{p_2}}\le C_{f}C_{A} (t-s)^{\gamma_2}.
\end{equation}

Since, by assumption, $\gamma_1+\gamma_2>1$, we get for any $m>n$
 \begin{equation}\label{mnraznost}
\|I_{m}-I_n\|_{\SL_h}\le \sum_{i=n}^{m-1}\|I_{i+1}-I_i\|_{\SL_h}\le C2^{-n(\gamma_1+\gamma_2-1)},
 \end{equation}
 where $C=C_{f}C_{A}(t-s)^{\gamma_1+\gamma_2}/(1-2^{-(\gamma_1+\gamma_2-1)})>0$. Taking in \eqref{mnraznost} $n=0$ and using \eqref{inol}, we deduce that $I_m\in\SL_h$. Furthermore, \eqref{mnraznost} implies that the sequence $(I_n)_{n\in\Z_+}$ is a Cauchy sequence in $\SL_h$. Hence it converges in $\SL_h$ to some $\wt I\in\SL_h$.
On the other hand, $I_n$ converges to $I$ in probability as $n\to\infty$. Therefore $I=\wt I$ a.s. and thus $I_n$ converges to $I$ in $\SL_h$.

(ii) Fix $n\in\Z_+$. Since $I_m$ converges to $I$ in $\SL_h$ as $m\to\infty$, we have
\begin{equation*}
\|I-I_n\|_{\SL_h}=\lim_{m\to\infty}\|I_m-I_n\|_{\SL_h}.
\end{equation*}
Combining this with \eqref{mnraznost}, we establish \eqref{Lhnormdirrem}.

(iii) Since $I_n$ converges to $I$ in $\SL_h$, we have
\begin{equation}\label{limlh}
\|I\|_{\SL_h}=\lim_{n\to\infty}\|I_n\|_{\SL_h}.
\end{equation}
Moreover, taking into account \eqref{a-la-Khoa} and \eqref{inol}, we deduce
\begin{align*}
\|I_n\|_{\SL_h}\le& \|I_0\|_{\SL_h} +\sum_{i=0}^{n-1}\|I_{i+1}-I_i\|_{\SL_h}\le C_{f}C_{A} (t-s)^{\gamma_2} +
 \frac{C_{f}C_{A}}{1-2^{-(\gamma_1+\gamma_2-1)}} (t-s)^{\gamma_1+\gamma_2}\\
 \le& CC_{f}C_{A}(t-s)^{\gamma_2},
\end{align*}
where $C=C(T,\gamma_1,\gamma_2)$. Combining this with \eqref{limlh} we obtain \eqref{Lhnormdir}.
\end{proof}

The second lemma of this subsection deals with the approximations of the integral with respect to a Dirichlet process.
\begin{Lemma}\label{L:poslednijboj}
Let $(X,A)$ be as in Definition~\ref{D:decompos}.
Let $(f_n)_{n\in\Z_+}$ be a sequence of functions $\R\to\R$ that are uniformly bounded, continuous, and have a bounded continuous derivative. Assume that all $f_n$ satisfy condition \eqref{condej1} with the same parameters $p_1$, $\gamma_1$, and $C_{f_1}$ for all $ n \in \Z_+$. 

Let $(b_n)_{n\in\Z_+}$ be a sequence of bounded continuous functions. Define
\begin{equation*}
A^n_t:=\int_0^t b_n(X_s)\,ds,\quad t\in[0,T].
\end{equation*}
Suppose that for each $t\in[0,T]$ the sequence $(A^n(t))_{n\in\Z_+}$ converges in probability to $A(t)$.

Assume that $A$ and all functions $A^n$, $n\in\Z_+$ satisfy condition \eqref{condej2} with the same parameters $C_A$, $p_2$, $\gamma_2$.

Finally, assume that $\gamma_1+\gamma_2>1$ and $1/p_1+1/p_2\le1$.

Then for any $t\in[0,T]$ we have
\begin{equation}\label{finalresult}
\int_0^t f_n(X_s)\,d A_s-\int_0^t f_n(X_s)b_n(X_s)\,d s\to 0,\,\,\text{in probability as $n\to\infty$.}
\end{equation}
\end{Lemma}

\begin{proof}
Fix $t\in[0,T]$, $\eps>0$. For brevity, let us denote
\begin{equation*}
I_n:= \int_0^t f_n(X_s)dA_s,\quad J_n:=\int_0^t f_n(X_s)b_n(X_s)ds.
\end{equation*}
For $m\in\Z_+$ let $D_m:=(t_m^k)_{k\in[0,2^m]}$ be the $m$-th dyadic partition of $[0,t]$, i.e., $t_m^k:=kt2^{-m}$.
Consider the corresponding Riemann sums with respect to $D_m$:
\begin{equation*}
I_{n,m}:= \sum_{i=0}^{2^m-1} f_n(X_{t^i_m})(A_{t^{i+1}_m}-A_{t^i_m}),\quad
J_{n,m}:= \sum_{i=0}^{2^m-1} f_n(X_{t^i_m})(A^n_{t^{i+1}_m}-A^n_{t^i_m}).
\end{equation*}
Clearly, for any $n,m\in\Z_+$
\begin{align}\label{otzenkaver}
\P(|I_n-J_n|>\eps)&\le\P(|I_n-I_{n,m}|>\eps/3)+\P(|I_{n,m}-J_{n,m}|>\eps/3)+\P(|J_{n,m}-J_n|>\eps/3)\nn\\
&\le \frac3\eps\E|I_n-I_{n,m}|+\frac3\eps\E|J_{n,m}-J_n|+\P(|I_{n,m}-J_{n,m}|>\eps/3).
\end{align}

We apply Lemma~\ref{L:mainyoung}(ii) to the functions $f_n$ and $A$. We get
\begin{equation}\label{Ipartsum}
\E|I_n-I_{n,m}|\le C C_{f_1} C_A2^{-m(\gamma_1+\gamma_2-1)},
 \end{equation}
for some universal constant $C>0$ that does not depend on $n$, $m$. In a similar way, applying Lemma~\ref{L:mainyoung}(ii) to the functions $f_n$ and $A_n$, we deduce
\begin{equation}\label{Jpartsum}
\E|J_n-J_{n,m}|\le C C_{f_1} C_A2^{-m(\gamma_1+\gamma_2-1)}.
\end{equation}
Denote $K:=\sup_{i\in\Z_+}\|f_i\|<\infty$. Note that
\begin{align*}
\P(|I_{n,m}-J_{n,m}|>\eps/3)&\le \P(K\sum_{i=0}^{2^m-1}(A_{t^{i+1}_m}-A^n_{t^{i+1}_m}-A_{t^i_m}+A^n_{t^i_m})>\eps/3)\\
&\le 2\sum_{i=0}^{2^m}\P(|A_{t^{i}_m}-A^n_{t^{i}_m}|>\frac{\eps}{6K2^m})\longrightarrow0\,\,\text{as $n\to\infty$}
\end{align*}
where used the convergence in probability of $A^n_s$ to $A_s$ for each fixed $s\in[0,T]$. Combining this with \eqref{otzenkaver}, \eqref{Ipartsum}, \eqref{Jpartsum}, we finally obtain
\begin{equation*}
\lim_{n\to\infty}\P(|I_n-J_n|>\eps)\le C \eps^{-1} C_{f_1} C_A 2^{-m(\gamma_1+\gamma_2-1)}.
\end{equation*}
Since $m$ is arbitrary, by passing to the limit as $m\to\infty$, we get \eqref{finalresult}.
\end{proof}

\section{Proof of Proposition \ref{p:resolvent}: analysis of the resolvent equation}\label{S:2}
The primary purpose of this section is to prove Proposition
\ref{p:resolvent}. We will follow an approach similar to
\cite{Pr12}. We begin the analysis of equation \eqref{me} with the
case $f=0$.

\begin{Lemma}\label{L:1}
The resolvent equation \eqref{me} with $\lambda\ge1$, $f=0$, $g\in\C^\eta$, $\eta>-\alpha$ has a unique solution in the class of bounded functions. Furthermore, for each $\gamma\in[0\vee\eta,\alpha+\eta)$, this solution $u^\lambda_{0,g}\in \C^\gamma$ and there exists a constant $C= C(\eta, \gamma)>0$ such that
\begin{equation}\label{estl1}
\|u^\lambda_{0,g} \|_\gamma \le C \lambda^{-1-\eta/\alpha+\gamma/\alpha}\|g\|_{\eta}.
\end{equation}

\end{Lemma}
\begin{proof}
 We begin with uniqueness.
Let $u_1,u_2$ be two bounded solutions of \eqref{me} with $\lambda>0$, $f=0$, $g\in\C^\eta$. Then the function $v:=u_1-u_2$ is obviously bounded and we have $\OL_{\alpha} v = \lambda v$. Take any test function $\phi\in\S$. It follows from the definition of the solution that
\begin{equation} \label{us2}
\la v,\OL_{\alpha}\phi-\lambda\phi\ra =0.
\end{equation}
By the Hile--Yosida theorem, see, e.g., \cite[Theorem 31.3 (i)]{Sato99}, for any $h\in \C_c^{\infty}$ there exists $\psi\in {\mathcal D}(\OL_{\alpha})$ such that
\begin{equation} \label{dense1}
 \OL_{\alpha}\psi -\lambda\psi = h.
\end{equation}
It follows from \cite[Theorem 31.5]{Sato99} that $\C_c^\infty$ is a core for $\OL_\alpha.$ By definition of a core, we have that for each
$\psi \in {\mathcal D}(\OL_\alpha)$ there exists a sequence $\psi_n \in \C_c^\infty$, $n\in\Z_+$ such that
\begin{equation} \label{dense2}
\snorm{\psi_n - \psi} \rightarrow 0 \mbox{ and }
\snorm{\OL_\alpha \psi_n - \OL_\alpha \psi} \rightarrow 0 \text{ as $n\to\infty$}.\end{equation}
Since $\C_c^\infty \subset \S$, we deduce from (\ref{dense1}) and (\ref{dense2}) that for any $h \in \C_c^\infty$ there exists $\psi_n \in S$ such that
\begin{equation} \label{dense3}
\lim_{n\to\infty}\snorm{ h - (\OL_\alpha \psi_n - \lambda \psi_n)}=0.
\end{equation}
Therefore \eqref{dense3}, \eqref{us2} and the dominated convergence theorem imply
 \begin{equation*}
\la v,h\ra =0
\end{equation*}
for all $h \in \C_c^\infty$. This yields $v = 0$ and completes the proof of uniqueness.

To show existence of solution and to establish estimate \eqref{estl1}
we adapt some ideas from the proof of \cite[Theorem~3.3]{Pr12}. Fix
$\gamma\in[0\vee\eta,\alpha+\eta)$ and take any $\alpha'\in(\alpha,2)$. We
 begin with the case $g\in\C^\infty_b$. It was shown in
 \cite[Theorem~3.3]{Pr12} that in this case equation \eqref{me} with
 $f=0$ has a unique solution in $\C^{\alpha'}$. This solution is
 given by
\begin{equation*}
u^{\lambda}_{0,g}(x):=\int_0^\infty e^{-\lambda t} P_t g(x)\, dt,\quad x\in\R,
\end{equation*}
where the semigroup $(P_t)_{t\ge0}$ is as in Definition (\ref{gensemgp}).
Hence, using Lemma~\ref{P:1}(ii, iii), we obtain
\begin{align}\label{firstrez}
\|u^{\lambda}_{0,g}\|_{\gamma}&\le\int_0^{+\infty} e^{-\lambda t} \|P_t g\|_{\gamma}\, dt=
\int_0^1 e^{-\lambda t} \|P_t g\|_{\gamma}\, dt+\int_1^{+\infty} e^{-\lambda t} \|P_t g\|_{\gamma}\, dt\nn\\
&\le
C \|g\|_{\eta}\int_0^1 e^{-\lambda t} t^{-(\gamma-\eta)/\alpha}\, dt+
C \|g\|_{\eta}\int_1^\infty e^{-\lambda t} \, dt\nn\\
&\le
C \lambda^{-1-\eta/\alpha+\gamma/\alpha}\|g\|_{\eta},
\end{align}
where the last inequality follows from the fact that $\lambda\ge1$.
Now take any $g\in\C^\eta$. Let $g_n\in\C^\infty_b$, $n\in\Z_+$ be a sequence of approximations of $g$ such that $\|g_n-g\|_\eta\to0$ as $n\to\infty$.

Consider the function $u_n:=u^{\lambda}_{0,g_n}$. By above, $u_n$ is well--defined and $u_n\in \C^{\alpha'}$.

Let $v_{n,m}:=u_n-u_m$, $n,m\in \Z_+$. We see that $v_{n,m}$ solves \eqref{me} with $f=0$ and the right--hand side $g_n-g_m$. Furthermore, since $u_n, u_m\in\C^{\alpha'}$, we see that $v_{n,m}\in\C^{\alpha'}$. Recall that the solution to \eqref{me} with $f=0$ and smooth right--hand side is unique in class $\C^{\alpha'}$ by \cite[Theorem~3.3]{Pr12}. Therefore we can apply \eqref{firstrez} to obtain
\begin{equation*}
\|u_n-u_m\|_{\gamma}=\|v_{n,m}\|_{\gamma}=\|u^{\lambda}_{0,g_n-g_m}\|_{\gamma}\le C \lambda^{-1-\eta/\alpha+\gamma/\alpha}\|g_n-g_m\|_{\eta}.
\end{equation*}
This implies that $(u_n)_{n\in\Z_+}$ is a Cauchy sequence in $\C^\gamma$ and hence there exists some $u\in \C^\gamma$ such that $\|u-u_n\|_{\gamma}\to0$ as $n\to\infty$. We claim that $u$ is a solution to \eqref{me} with $f=0$ and the right--hand side $g$. Indeed,
\begin{align*}
\|\lambda u -\OL_{\alpha} u - g\|_{\gamma-\alpha}&= \|\lambda (u-u_n) -\OL_{\alpha} (u-u_n) - (g-g_n)\|_{\gamma-\alpha}\\
&\le \lambda\| u-u_n\|_{\gamma-\alpha}+ \|\OL_{\alpha} (u-u_n)\|_{\gamma-\alpha} + \|g-g_n\|_{\gamma-\alpha}\\
&\le \lambda\| u-u_n\|_{\gamma}+ \|u-u_n\|_{\gamma} + \|g-g_n\|_{\eta},
\end{align*}
where we used the fact that $\gamma-\alpha\le \eta$ , Lemma~\ref{L:31}, and Lemma~\ref{P:1}. By passing to the limit as $n\to\infty$ we deduce
\begin{equation*}
\lambda u -\OL_{\alpha} u - g=0,
\end{equation*}
and hence $u$ indeed solves \eqref{me} with $f=0$ and the right--hand side $g$. To complete the proof it remains to note that
\begin{equation*}
\|u\|_{\gamma}\le \|u_n\|_{\gamma}+\|u-u_n\|_{\gamma}\le C \lambda^{-1-\eta/\alpha+\gamma/\alpha}\|g_n\|_{\eta}+\|u-u_n\|_{\gamma}.
\end{equation*}
Again by passing to the limit as $n\to\infty$, we obtain \eqref{estl1}.
\end{proof}

Now we are ready to prove the first part of Proposition \ref{p:resolvent}.
\begin{proof}[Proof of Proposition \ref{p:resolvent}(i)]
We begin by proving a crucial inequality that will be used many times in this proposition. Let $\delta>-\alpha$ and
let $u$ be any bounded solution to \eqref{me} with $f,g\in\C^{\delta}$, $\lambda \ge 1$. Obviously
\begin{equation}\label{otherform}
\lambda u -\OL_{\alpha} u = f u'+ g.
\end{equation}
Using \eqref{otherform} and Lemma~\ref{L:1}, we derive that for any $\gamma\in[0\vee\delta,\alpha+\delta)$ there exists $C=C(\delta,\gamma )>0$ such that
\begin{align}\label{step1}
\|u \|_{\gamma}&\le C\lambda^{-1-\frac\delta\alpha+\frac\gamma\alpha}\| f u'+ g\|_\delta\le
C\lambda^{-1-\frac\delta\alpha+\frac\gamma\alpha}(\|g\|_\delta+\|f\|_{\delta}\|u'\|_{\rho})\nn\\
&\le C\lambda^{-1-\frac\delta\alpha+\frac\gamma\alpha}(\|g\|_\delta+\|f\|_{\delta}\|u\|_{\rho+1}),
\end{align}
whenever $\delta+\rho>0$ and $\rho\ge\delta$. Here we have used inequalities from Section \ref{S:11}, specifically \eqref{product} in the second inequality and \eqref{der} in the third inequality. We will apply (\ref{step1}) repeatedly. Now we can start proving \eqref{est}.

First we deal with the case $\eta>0$. By \cite[Theorem~3.4]{Pr12} equation \eqref{me} has a solution $u^{\lambda}_{f,g}\in\C^{\alpha}$ for all $\lambda>0$. We will show that
\begin{equation}\label{above}
u^{\lambda}_{f,g}\in\C^{\gamma}\,\,\text{ for any $\lambda>0$, $\gamma<\alpha+\eta$}.
\end{equation}
Assume the converse. Then there exist $\alpha\le \gamma_1<\gamma_2<\alpha+\eta$, such that $\|u^{\lambda}_{f,g}\|_{\gamma_1}<\infty$, $\|u^{\lambda}_{f,g}\|_{\gamma_2}=\infty$ and $\gamma_2-\gamma_1<\alpha-1$. We apply \eqref{step1} with $\gamma=\gamma_2$, $\delta=(\gamma_1-1)\wedge\eta$, $\rho=\gamma_1-1$. Note that all the additional constraints are satisfied: since $\gamma_2< \alpha+\eta$ and $\gamma_2< \alpha+\gamma_1-1$, we see that
$\gamma_2< \alpha+(\gamma_1-1)\wedge\eta$ and thus $\gamma_2\in[0\vee\delta,\alpha+\delta)$. Using
the fact that by assumption $f,g\in\C^\eta$, we derive
\begin{equation*}
\|u^{\lambda}_{f,g} \|_{\gamma_2}\le C(\lambda) (\|g\|_{(\gamma_1-1)\wedge\eta}+\|f\|_{(\gamma_1-1)\wedge\eta}\|u^{\lambda}_{f,g}\|_{\gamma_1})<\infty,
\end{equation*}
However this contradicts the assumption $\|u^{\lambda}_{f,g}\|_{\gamma_2}=\infty$. Thus $u^{\lambda}_{f,g}\in\C^{\gamma}$ for any $\gamma<\alpha+\eta$.

We apply \eqref{step1} again, but now with a different set of parameters: we take $\gamma=\eta+1$, $\delta=\rho=\eta$. Then we obtain
\begin{equation}\label{step2eta}
\|u^{\lambda}_{f,g} \|_{\eta+1}\le C_1\lambda^{-1+\frac1\alpha}(\|g\|_\eta+\|f\|_{\eta}\|u^{\lambda}_{f,g}\|_{\eta+1}),
\end{equation}
where $C_1>0$. By above, $\|u^{\lambda}_{f,g} \|_{\eta+1}<\infty$. Take now $\lambda_1\ge 1$ such that
\begin{equation*}
1- C_1\lambda_1^{-1+\frac1{\alpha}} M \ge 1/2.
\end{equation*}
Since $C_1$ depends only on $\eta$, we see that $\lambda_1$ depends only on $\eta, M$. For $\lambda\ge \lambda_1$  and $\|f\|_{\eta}\leq M$ we get from \eqref{step2eta}
\begin{equation}\label{step3eta}
\|u^{\lambda}_{f,g}\|_{\eta+1}\le 2C_1\lambda^{-1+\frac1\alpha}\|g\|_\eta\le 2C_1\|g\|_\eta.
\end{equation}
Finally, applying again \eqref{step1} with $\delta=\rho=\eta$ and $\gamma\in[\eta,\alpha+\eta)$ we get for $\lambda\ge \lambda_1$
\begin{equation*}
\|u^{\lambda}_{f,g}\|_{\gamma}\le C\lambda^{-1-\frac\eta\alpha+\frac\gamma\alpha}\|g\|_\eta(\|f\|_\eta+1),
\end{equation*}
where we used bound \eqref{step3eta}. This establishes \eqref{est} for $\eta>0$ when $\lambda \ge \lambda_1$.

Now we can treat the case $\eta\le0$. The problem here is that we do not know a priori that in this case \eqref{me} has a solution. Therefore we have to study approximations. We will still use \eqref{step1} as a main tool however with a different set of parameters (since $\eta$ is negative we cannot take $\delta=\rho=\eta$).

Thus, we start with considering any $f,g\in\C_b^\infty$, $\lambda>0$. By \eqref{above}, equation \eqref{me} has a solution
$u^\lambda_{f,g}\in\C_b^\infty$. We apply \eqref{step1} with $\gamma=\alpha/2+1/2$, $\delta=\eta$, $\rho=\alpha/2-1/2$.
Since $\eta>1/2-\alpha/2$ one can easily see that all the additional constraints on the parameters in \eqref{me} are satisfied. We get
\begin{equation*}
\|u^{\lambda}_{f,g} \|_{(1+\alpha)/2}(1- C_2\lambda^{-\frac12+\frac1{2\alpha}-\frac\eta\alpha}\|f\|_\eta)\le
C\lambda^{-\frac12+\frac1{2\alpha}-\frac\eta\alpha}\|g\|_\eta,
\end{equation*}
where $C_2>0$ and we have used the fact that $\|u^{\lambda}_{f,g} \|_{(1+\alpha)/2}<\infty$. Choose $\lambda_2\ge 1$ such that
\begin{equation*}
1- C_2\lambda_2^{-\frac12+\frac1{2\alpha}-\frac\eta\alpha}M \ge 1/2.
\end{equation*}
Similarly, $\lambda_2$ depends only $\eta, M$. For $\lambda\ge \lambda_2 = \lambda_2(\eta, M)$ and $\|f\|_{\eta}\leq M$ we get
\begin{equation}\label{rez1}
\|u^{\lambda}_{f,g} \|_{(1+\alpha)/2}\le 2C_2 \|g\|_\eta.
\end{equation}

Now we take any $f,g\in\C^\eta$. Similar to the proof of Lemma~\ref{L:1}, we
approximate $f$ and $g$ by the sequences $f_n,g_n\in\C_b^\infty$, correspondingly, such that
\begin{align*}
&\|f_n-f\|_\eta\to0,\, \|g_n-g\|_\eta\to0,\,\text{ as $n\to\infty$}\\
&\|f_n\|_\eta\le 2 \|f\|_\eta,\, \|g_n\|_\eta\le 2 \|g\|_\eta\,\quad \text{for all $n\in\Z_+$}.
\end{align*}

Consider the function $u_n\!:=u^\lambda_{f_n,g_n}$. By above, $u_n\in\C^\infty_b$. Put $v_{n,m}\!:=u_n-u_m$, $n,m\!\in\Z_+$.
It follows that $v_{n,m}\in\C^\infty_b$ and solves
\begin{equation}\label{mutant}
\lambda v_{n,m} -\OL_{\alpha} v_{n,m} - f_n v_{n,m}'= g_n-g_m+(f_m-f_n)u_m'.
\end{equation}
Clearly, the right--hand side of \eqref{mutant} is in $\C^{\infty}_b$. Therefore, using uniqueness theorem for equation \eqref{me} with smooth coefficients (\cite[Theorem~3.4]{Pr12}) we see that $v_{n,m}$ is the unique solution to \eqref{mutant}. Thus, we can apply bound \eqref{rez1}. We make use of the fact that $\|f_n\|_\eta\le 2 \|f\|_\eta$ to get for $\lambda\ge
\lambda_2(\eta,2 M)$ and $\|f\|_\eta \leq M$
\begin{align*}
\|u_n-u_m \|_{(1+\alpha)/2}&=\|v_{n,m} \|_{(1+\alpha)/2}\le C (\|g_n-g_m\|_\eta
+\|(f_m-f_n)u_m'\|_{\eta})\\
&\le C(\|g_n-g_m\|_\eta +\|f_m-f_n\|_{\eta}\|u_m'\|_{\alpha/2-1/2})\\
&\le C (\|g_n-g_m\|_\eta
+\|f_m-f_n\|_{\eta}\|g_m\|_{\eta}),
\end{align*}
where in the final inequality we used bound \eqref{rez1} once again. Recalling that $\|g_m\|_\eta\le 2 \|g\|_\eta$, we see that the sequence $(u_n)_{n\in\Z_+}$ is a Cauchy sequence in $\C^{(1+\alpha)/2}$. Hence there exists $u\in\C^{(1+\alpha)/2}$ such that
$\|u_n-u\|_{(1+\alpha)/2}\to0$ as $n\to\infty$. Applying Lemma~\ref{P:1}  we derive for any $n\in\Z_+$
\begin{align*}
\|\lambda u -\OL_{\alpha} u - &f u'- g\|_{(1-\alpha)/2}=\|\lambda (u-u_n) -\OL_{\alpha} (u-u_n) - f (u-u_n)'- (g-g_n)\|_{(1-\alpha)/2}\\
&\le \lambda\| u-u_n\|_{(1+\alpha)/2}+C\|u-u_n\|_{(1+\alpha)/2}+C\| f\|_\eta \|u-u_n\|_{(1+\alpha)/2}+C \|g-g_n\|_{\eta}.
\end{align*}
After taking the limit as $n\to\infty$, we see that $u$ solves \eqref{me}.

Note that thanks to \eqref{rez1}
\begin{equation*}
\|u\|_{(1+\alpha)/2}\le \lim_{n\to\infty} \|u-u_n\|_{(1+\alpha)/2}+\limsup_{n\to\infty}\|u_n\|_{(1+\alpha)/2}\le C\|g\|_\eta.
\end{equation*}
Using this inequality and the fact that $u$ solves \eqref{me}, we apply again \eqref{step1} with ${\gamma\in[0,\alpha+\eta)}$,
$\delta=\eta$, $\rho= \alpha/2-1/2$ to obtain
\begin{equation*}
\|u \|_{\gamma}\le C\lambda^{-1-\frac\eta\alpha+\frac\gamma\alpha}(\|g\|_\eta+C \|f\|_{\eta}\|g\|_\eta).
\end{equation*}
This establishes \eqref{est} for $\eta \leq 0$ when $\lambda > \lambda_2(\eta, 2M)$.
Set
 \[ \lambda_0(\eta, M): = \begin{cases} \lambda_1(\eta,M) & \mbox{ if } \eta >0 \\ \lambda_2(\eta,2M) & \mbox{ if } \eta \leq 0 \end{cases} \] to complete the proof of (\ref{est}).

Thus it remains to show uniqueness. If $u_1,u_2\in \C^{(1+\alpha)/2}$ are two solutions of \eqref{me}, then for $v:=u_1-u_2$ we obviously have $v\in \C^{(1+\alpha)/2}$ and
\begin{equation*}
\lambda v-\OL_{\alpha} v =f v'.
\end{equation*}
Therefore the right--hand side of the above equation is well-defined and is in $\C^\eta$. Therefore, we can apply  \eqref{step1} with $\gamma=\alpha/2+1/2$, $\delta=\eta$, $\rho= \alpha/2-1/2$ to obtain
\begin{equation*}
\|v\|_{(1+\alpha)/2}\le C\lambda^{-\frac12+\frac1{2\alpha}-\frac\eta\alpha}\|f\|_\eta \|v\|_{(1+\alpha)/2}\le\frac12
\|v\|_{(1+\alpha)/2},
\end{equation*}
whenever $\lambda\ge\lambda_0(\eta, M)$ and $\|f\|_\eta \leq M$. As $\|v\|_{(1+\alpha)/2} < \infty$, this implies that $\|v\|_{(1+\alpha)/2}=0$ and hence $v=0$. This establishes uniqueness of the solutions to \eqref{me}.
\end{proof}

\begin{proof}[Proof of Proposition \ref{p:resolvent} (ii)] Denote $u:=u^{\lambda}_{f,g}$, $u_n:=u^{\lambda_n}_{f_n,g_n}$, $v_n:=u-u_n$. By part (i), $v_n\in\C^{(1+\alpha)/2}$ and solves
 \begin{equation*}
 \lambda v_n -\OL_{\alpha} v_n - f v_n'= g_n-g+u_n'(f-f_n).
 \end{equation*}
 Therefore, by Proposition \ref{p:resolvent}(i) for any $\lambda\ge \lambda_0(\eta , M)$
 \begin{align*}
 \|u_n-u \|_{(1+\alpha)/2}&=\|v_n \|_{(1+\alpha)/2}\le C\lambda^{-\frac12+\frac1{2\alpha}-\frac\eta\alpha} (\|g_n-g\|_\eta+
 \|f_n-f\|_\eta\|u_n \|_{(1+\alpha)/2}) (1+ \|f \|_\eta)\\
 &\le C(1+M)^2\lambda^{-\frac12+\frac1{2\alpha}-\frac\eta\alpha} (\|g_n-g\|_\eta+
 \|f_n-f\|_\eta\|g_n \|_{(1+\alpha)/2})\to0.\qedhere
 \end{align*}
 \end{proof}

\section{Proof of Theorem \ref{T:1}}
\label{sec:5}
In this section we present the proof of our main result, Theorem
\ref{T:1}. We will follow the
sketch of the proof presented in Section~\ref{sec:overview}. We will
rely on the machinery related to the resolvent equation developed in
Section~\ref{S:2} and Preliminaries from Section~\ref{S:prelimka}. Recall that we have
fixed $\alpha\in(1,2)$, $\beta\in(1/2-\alpha/2,0)$, $b\in\C^\beta$. the initial condition $x\in\R$ and the length of time interval $T>0$. Our goal
is to show that \eqref{mainsde} has a unique strong solution on time
interval $[0,T]$.

We begin with a very standard calculation of a second moment of a stochastic integral. We will use this result a couple of times and hence for the sake of completeness we decided to state it precisely.

\begin{Lemma}\label{L:vtorojmoment}
Let $L$ be an $\alpha$-stable L\'evy process, $\nu$ be its L\'evy measure and $\wt N$ be the compensated Poisson measure associated with $L$. Let $f\colon [0,T]\times\R_+\times\Omega\to\R$ be a measurable function adapted to the filtration of $L$. Suppose that there exist $\gamma\in(\alpha/2,1]$ and constant $C_f>0$ such that $\P$--a.s.
\begin{equation}\label{fbound}
|f(s,r,\omega)|\le C_f(|r|^\gamma\wedge1),\quad s\in[0,T],\, r\in\R.
\end{equation}
Then there exists a constant $C=C(T)>0$ such that for any stopping times $\tau_1,\tau_2\in[0,T]$ with $\tau_1\le \tau_2$ we have
\begin{equation}\label{secondmomentstoch}
\E \Bigl(\int_{\tau_1}^{\tau_2} \int_\R f(s,r,\omega)\wt N(ds,dr) \Bigr)^2\le C C_f^2 \E |\tau_2-\tau_1|.
\end{equation}
\end{Lemma}

The proof of the lemma is given in Appendix~\ref{A:2mom}.

The upcoming subsections are devoted to the proofs of Propositions~\ref{p:vtreal}, \ref{p:wex}, \ref{p:pu}. We complete the proof of Theorem \ref{T:1} in Section~\ref{sec:kur}.

\subsection{Proof of Proposition~\ref{p:vtreal}: any weak solution of (\ref{mainsde}) solves the Zvonkin equation}\label{SS:vtreal}
 We will use different properties of integrals with respect to the Dirichlet processes established in Section~\ref{S:Dirya}. We begin with the following simple moment bound.

 \begin{Lemma} \label{L:Xmoments}
Let $X$ be a weak solution of SDE \eqref{mainsde}. Then for any $\gamma\in[0,\alpha)$ there exists a constant $C>0$ such that for any $s,t\in[0,T]$
\begin{equation}\label{momentX}
\E|X_t-X_s|^\gamma\le C|t-s|^{\gamma/\alpha}.
\end{equation}
\end{Lemma}
\begin{proof}

First note that by basic properties of an $\alpha$--stable process we have for $\gamma\in[0,\alpha)$
\begin{equation*}
\E|L_t-L_s|^\gamma\le C |t-s|^{\gamma/\alpha},\quad s,t\in[0,T].
\end{equation*}
Note that by Definition~\ref{D:sol} for any $\kappa<1+\beta/\alpha$. So for $s,t\in[0,T]$ using Jensen's inequality we have
\begin{equation}\label{MomentboundX}
\E|X_t-X_s|^\gamma\le C\E|A_t-A_s|^\gamma+C\E|L_t-L_s|^\gamma\le C |t-s|^{\gamma\kappa}+
C |t-s|^{\gamma/\alpha}.
\end{equation}
Since $1/\alpha<1+\beta/\alpha$, we can take in \eqref{MomentboundX} $\kappa=1/\alpha$. This immediately yields \eqref{momentX}.
\end{proof}

\begin{Lemma} \label{L:smooth}
Let $X$ be a weak solution of SDE \eqref{mainsde}. Then for any $\gamma\in[0,\frac1{1/2-\beta/\alpha})$, $\kappa<1+\beta/\alpha$ there exists a constant $C>0$ such that for any $f\in\C_b^\infty$, $0\le s\le t \le T$ we have
\begin{equation}\label{momentforAn}
 \E\Bigl(\int_s^t f(X_l)\,dl\Bigr)^\gamma\le C |t-s|^{\gamma\kappa}\|f\|_\beta^\gamma.
\end{equation}

\end{Lemma}
\begin{proof}
We begin by observing that $X$ is a Dirichlet process (as $A$ has zero energy due to \eqref{3cond}). Let us fix $f\in\C_b^\infty$ and $0\le s<t\le T$. Take now any $\gamma<1/(1/2-\beta/\alpha)$.
Since we took $\beta<0$, we have $\gamma<2$.

 We claim that it is sufficient to show \eqref{momentforAn} only for those $s$, $t$ that are close enough; further we assume that $|t-s|\le 1$. Indeed, if this is already proven, then for any $s,t\in[0,T]$, $s\le t$ we can take an increasing sequence $(t_i)_{i\in[0,N]}$ such that $t_0=s$, $t_N=t$, $t_{i+1}-t_i\le 1$ and
$N \le T+1$. Then
\begin{align*}
 \E\Bigl(\int_s^t f(X_l)\,dl\Bigr)^\gamma&\le N \sum_{i=0}^{N-1}\E (\int_{t_i}^{t_{i+1}} f(X_l)\,dl\Bigr)^\gamma
 \le
 C N\|f\|_\beta^\gamma \sum_{i=0}^{N-1} |t_{i+1}-t_{i}|^{\gamma\kappa}\\
 &\le C\|f\|_\beta^\gamma N^{1+\gamma\kappa}|t-s|^{\gamma\kappa}\\
 &\le C\|f\|_\beta^\gamma |t-s|^{\gamma\kappa},
\end{align*}
for some $C=C(T)$. Thus, we can safely assume that $|t-s|\le 1$.

Now let us consider a function $v^\lambda:=u^\lambda_{0,f}$, where $\lambda \ge1$; this function is well-defined by Lemma~\ref{L:1}. We apply It\^{o}'s formula
for Dirichlet processes \cite[Theorem~3.4]{CJMS} (see also \cite[Theorem 5.15(ii)]{BR17}) to derive for $0\le s< t\le T$
\begin{align*}
v^\lambda(X_t)-v^\lambda(X_s)=&\int_s^t\OL_{\alpha} v^\lambda(X_l) \,dl
 +\int_s^t\int_\R [ v^\lambda(X_{l-} + r) - v^\lambda(X_{l-})]  \wt N(dl, dr)\\
 &+\int_s^t(v^\lambda)'(X_{l-}) \,dA_l\\
 =&\int_s^t\int_\R [ v^\lambda(X_{l-} + r) - v^\lambda(X_{l-})]  \wt N(ds, dr)\\
 &+\lambda \int_s^t v^\lambda(X_l)\,dl -\int_s^t f(X_l)\,dl\\
 &+\int_s^t (v^\lambda)'(X_l) \,dA_l,
\end{align*}
where we also used the fact that $v^\lambda$ solves \eqref{me} with $0$ in place of $f$ and $f$ in place of $g$. By rearranging the terms we get
\begin{align}\label{itov}
\E\Bigl(\int_s^t f(X_l)\,dl\Bigr)^\gamma\le& C\snorm{v^\lambda}^\gamma(2+\lambda(t-s))^\gamma
+C\E \Bigl(\int_s^t\int_\R [ v^\lambda(X_{l-} + r) - v^\lambda(X_{l-})]  \wt N(ds, dr)\Bigr)^\gamma\nn\\
 &+C\E \Bigl(\int_s^t (v^\lambda)'(X_l) \,dA_l\Bigr)^\gamma\nn\\
 =&I_1+I_2+I_3.
\end{align}

Now
recall that we supposed
that $|t-s|\le 1$. Then we can take $\lambda:=(t-s)^{-1}$. We immediately get by Lemma \ref{L:1} that for any $\eps>0$
\begin{equation}\label{I1forb}
I_1\le C (t-s)^{\gamma+\gamma\beta/\alpha-\eps}\|f\|_\beta^\gamma.
\end{equation}
Note that for any $\rho\in(0,1)$
\begin{equation*}
|v^\lambda(X_{l-} + r) - v^\lambda(X_{l-})|\le 2\|v^\lambda\|_\rho(|r|^\rho\wedge1).
\end{equation*}
Therefore, we take $\rho=\alpha/2+\alpha\eps/\gamma$ and apply Lemma~\ref{L:vtorojmoment} with Jensen's inequality and again Lemma~\ref{L:1}. We get
\begin{equation}\label{I2forb}
I_2\le C (\|v^\lambda\|_\rho)^\gamma |t-s|^{\gamma/2}\le C |t-s|^{\gamma+\gamma\beta/\alpha-\eps}\|f\|_\beta^\gamma.
\end{equation}

Thus it remains to estimate $I_3$. Let $\rho\in(0,\alpha+\beta-1)$ and $\sigma\in(0,\alpha/\rho)$ be parameters to be chosen later. Then Lemma~\ref{L:1} and Lemma~\ref{L:Xmoments} imply for any $t_1,t_2\in[0,T]$
\begin{equation*}
\E| (v^\lambda)'(X_{t_1})- (v^\lambda)'(X_{t_2})|^\sigma\le
(\|(v^\lambda)'\|_{\rho})^\sigma\E|X_{t_1}-X_{t_2}|^{\rho\sigma}\le
C_1\|f\|_\beta^\sigma|t_1-t_2|^{\rho\sigma/\alpha}
\end{equation*}
for some $C_1>0$. Recall that by definition of solution, for any $\eps>0$ there exists $C_2>0$ such that for any $t_1,t_2\in[0,T]$
\begin{equation*}
\E|A_{t_1}-A_{t_2}|^2\le C_2|{t_1}-{t_2}|^{2(1+\beta/\alpha-\eps)}.
\end{equation*}
Thus, we can apply Lemma~\ref{L:mainyoung} with $C_f=C_1^{1/\sigma}\|f\|_\beta$, $C_A=C_2^{1/2}$, $p_1=\sigma$, $p_2=2$, $\gamma_1=\rho/\alpha$, $\gamma_2=1+\beta/\alpha-\eps$. Now we choose any
\begin{equation}\label{choicerhosigma}
\rho\in(-\beta+\eps\alpha,\alpha+\beta-1);\,\,\,\sigma\in(0,\alpha/\rho).
\end{equation}
This is possible for $\eps>0$ small enough since $-\beta<\alpha+\beta-1$ thanks to our assumption on $\beta$. Our choice of parameters \eqref{choicerhosigma} automatically implies that $\gamma_1+\gamma_2>1$. Further we also have that
\begin{equation*}
\frac1h:=\frac1{p_1}+\frac1{p_2}=\frac1{\sigma}+\frac1{2}>1/2-\beta/\alpha+\eps.
\end{equation*}
Recall that $\sigma$ and $\rho$ were arbitrary parameters that satisfy bounds in \eqref{choicerhosigma}. By choosing $\rho$ close enough to its lower bound $-\beta+\eps\alpha$ and choosing $\sigma$ close enough to its upper bound $\alpha/\rho$, one can make $1/h$ to be arbitrarily close (though still bigger) to $1/2-\beta/\alpha+\eps$. Since for small enough $\eps>0$ we have $1/2-\beta/\alpha+\eps<1$, we see that all the conditions of Lemma~\ref{L:mainyoung} are satisfied. Bound \eqref{Lhnormdir} yields for any $h\le (1/2-\beta/\alpha+\eps)^{-1}$
\begin{equation*}
\E\Bigl|\int_s^t v'(X_l) \,dA_l\Bigr|^h\le C \|f\|_\beta^h|t-s|^{h(1+\beta/\alpha-\eps)}.
\end{equation*}
Combining this with \eqref{I1forb} and \eqref{I2forb} and substituting them into \eqref{itov}, we obtain \eqref{momentforAn}. This completes the proof of the lemma.
\end{proof}

\begin{proof}[Proof of Proposition \ref{p:vtreal}]
Let $X$ be a weak solution to \eqref{mainsde}.
Let $(b_n)_{n\in\Z_+}$ be a sequence of $\C_b^\infty$ functions converging to $b$ in $\C^\beta$. Without loss of generality, we can assume that for each $n\in\Z_+$ we have $\|b_n\|_\beta\le2\|b\|_\beta$.

Fix $\lambda\ge\lambda_0(\beta,2\|b\|_\beta)$. Let $u_n = u^\lambda_{b_n,b_n}$ be a unique $\C^{1/2+\alpha/2}$ solution to \eqref{me}. It follows from Proposition \ref{p:resolvent}(i) that $u_n\in\C_b^\infty$. Definition~\ref{D:sol} implies that $A$ has zero energy and thus $X$ is a Dirichlet process. We apply It\^{o}'s formula for Dirichlet processes \cite[Theorem~3.4]{CJMS} (see also \cite[Theorem 5.15(ii)]{BR17}) to derive for $t\ge0$
\begin{align}\label{s1}
u_n(X_t)=&u_n(x)+\int_0^t\OL_{\alpha} u_n(X_s) \,ds
 +\int_0^t\int_\R [ u_n(X_{s-} + r) - u_n(X_{s-})]  \wt N(ds, dr)\nonumber\\
 &+\int_0^tu'_n(X_{s-}) \,dA_s.
\end{align}

We continue \eqref{s1} as follows, using the fact that $u_n$ solves \eqref{me}:
\begin{align}\label{s2}
u_n(X_t)=&u_n(x)+\int_0^t\int_\R [ u_n(X_{s-} + r) - u_n(X_{s-})]  \wt N(ds, dr)\nn\\
 &+\lambda \int_0^t u_n(X_s)\,ds -\int_0^t b_n(X_s)\,ds\nn\\
 &+\int_0^tu'_n(X_s) \,dA_s-\int_0^tu'_n(X_s) b_n(X_s)\,ds.
\end{align}
For any fixed $t\in[0,T]$ let us pass to the limit in \eqref{s2} as $n\to\infty$. Since $u_n$ converges to $u$ in $\C^{1/2+\alpha/2}$, it is clear that
\begin{equation}\label{pervlim}
u_n(X_t)\to u(X_t),\,\, u_n(x)\to u(x),\,\,\lambda \int_0^t u_n(X_s)\,ds\to\lambda \int_0^t u(X_s)\,ds,\,\,\text{a.s. as $n\to\infty$.}
\end{equation}
Note that since $1/2+\alpha/2>1$
\begin{equation*}
|u_n(X_{s-} + r) - u(X_{s-} + r)-u_n(X_{s-})+u(X_{s-})|\le \|u_n-u\|_{1/2+\alpha/2}(|r|\wedge1).
\end{equation*}
Therefore by Lemma~\ref{L:vtorojmoment} we have
\begin{equation}\label{vtomomeint}
\E\Bigl(\int_0^t\int_\R [ u_n(X_{s-} + r) - u(X_{s-} + r)-u_n(X_{s-})+u(X_{s-})]  \wt N(ds, dr)\Bigr)^2\le C
\|u_n-u\|_{1/2+\alpha/2}^2 T.
\end{equation}
By Proposition \ref{p:resolvent}(ii), $\|u_n-u\|_{1/2+\alpha/2}\to0$ as $n\to\infty$. Hence \eqref{vtomomeint} implies
\begin{equation}\label{vtorlim}
\int_0^t\int_\R [ u_n(X_{s-} + r) - u_n(X_{s-})]  \wt N(ds, dr) \to\int_0^t\int_\R [ u(X_{s-} + r) - u(X_{s-})]  \wt N(ds, dr),
\end{equation}
in probability as $n\to\infty$.

By the definition of a solution,
\begin{equation}\label{Alim}
\int_0^t b_n(X_s)\,ds \to A_t,
\end{equation}
in probability as $n\to\infty$.

Thus, it remains to find the limit of the last two terms in the right--hand side of \eqref{s2}. Fix $\eps>0$ small enough. Arguing exactly as in the proof of Lemma~\ref{L:smooth}, we see that for any $t_1, t_2\in[0,T]$
\begin{equation*}
\E| u_n'(X_{t_1})-u_n'(X_{t_2})|^\sigma\le
(\|u_n'\|_{\rho})^\sigma\E|X_{t_1}-X_{t_2}|^{\rho\sigma}\le
C_1\|b\|_\beta^\sigma|t_1-t_2|^{\rho\sigma/\alpha}
\end{equation*}
whenever
\begin{equation}\label{choiserhoepisode2}
\rho\in(-\beta+\eps\alpha,\alpha+\beta-1);\,\,\,\sigma\in(0,\alpha/\rho).
\end{equation}
Furthermore, by the definition of the solution and by Lemma~\ref{L:smooth}, for any $t_1, t_2\in[0,T]$
\begin{equation*}
\E|A_{t_1}-A_{t_2}|^h\le C_2 |t_1-t_2|^{h(1+\beta/\alpha-\eps)};\quad\E\Bigl|\int_{t_1}^{t_2} b_n(X_l)\,dl\Bigr|^h\le C_2 |t_1-t_2|^{h(1+\beta/\alpha-\eps)},
\end{equation*}
whenever $0\le h<1/(1/2-\beta/\alpha)$. Now we can apply Lemma~\ref{L:poslednijboj} to the functions $(u_n')$, $(b_n)$, $A$, with the following parameters:
$C_{f_1}:=C_1^{1/\sigma}\|b\|_\beta$, $p_1:=\sigma$, $\gamma_1:=\rho/\alpha$, $C_A:=C_2^{1/h}$, $p_2=h$, $\gamma_2=1+\beta/\alpha-\eps$. It follows from \eqref{choiserhoepisode2}, that $\gamma_1+\gamma_2>1$. Note that
\begin{equation*}
\frac{1}{p_1}+ \frac{1}{p_2}=\frac{1}{\sigma}+ \frac{1}{h}>\frac{\rho}{\alpha}+1/2-\beta/\alpha>1/2-2\beta/\alpha+\eps.
\end{equation*}
Further, by choosing $h$ close enough to $1/(1/2-\beta/\alpha)$, $\rho$ close enough to $-\beta+\eps\alpha$, $\sigma$ close enough to $\alpha/\rho$, one can guarantee that $\frac{1}{p_1}+ \frac{1}{p_2}$ will be arbitrarily close to $1/2-2\beta/\alpha+\eps$. However for $\eps$ small enough we have $1/2-2\beta/\alpha+\eps<1$. Hence for some suitable choice of parameters, one has $\frac{1}{p_1}+ \frac{1}{p_2}<1$. Hence all the conditions of  Lemma~\ref{L:poslednijboj} are satisfied. Thus,
\begin{equation*}
\int_0^tu'_n(X_s) \,dA_s-\int_0^tu'_n(X_s) b_n(X_s)\,ds\to0,
\end{equation*}
in probability as $n\to\infty$. Combining this with \eqref{pervlim}, \eqref{vtorlim}, \eqref{Alim} we can pass to the limit in Probability in
 \eqref{s2} as $n\to\infty$. We obtain that for each fixed $t$ the following identity holds a.s.
\begin{align*}
u(X_t)=&u(x)+\int_0^t\int_\R [ u(X_{s-} + r) - u(X_{s-})]  \wt N(ds, dr)\nn\\
 &+\lambda \int_0^t u(X_s)\,ds -A_t.
\end{align*}
To complete the proof it remains to note that $A_t=X_t-L_t-x$; thus $X_t$ is indeed a weak solution to equation \eqref{usde}.
\end{proof}

\subsection{Proof of Proposition~\ref{p:wex}: weak existence}
\label{subseq_we}

In this section we establish Proposition~\ref{p:wex}.

As explained in Section~2.1 we will construct a sequence of solutions to the approximated equations with smooth coefficients and then prove that this sequence has a limiting point, which solves SDE \eqref{mainsde} in the weak sense. Thus, let $(b_n)_{n\in\Z_+}$ be a sequence of $\C_b^\infty$ functions converging to $b$ in $\C^\beta$. Suppose that $\|b_n\|_\beta\le 2\|b\|_\beta$. Recall the definitions of $(X^n)_{n\in\Z_+}$ and $(A^n)_{n\in\Z_+}$, which are given in \eqref{appr} and \eqref{Ant}, correspondingly. Recall the definition of the function $\lambda_0$ in Proposition~\ref{p:resolvent}.

For $\lambda\ge\lambda_0(\beta,2\|b\|_\beta)$ let $u_n^\lambda:=u^\lambda_{b_n,b_n}$  be the unique
solution of the resolvent equation \eqref{me} in class $\C^{\frac{1+\alpha}{2}}$. By Proposition
\ref{p:resolvent}(i) $u_n^\lambda$ is well-defined and $u^\lambda_n\in\C^\infty_b$. For brevity, in this subsection further we will write just
$\lambda_0$ instead of $\lambda_0(\beta,2\|b\|_\beta)$.

\begin{Lemma}\label{L:apseq}
\begin{itemize}
 \item[{\rm(i)}] For each $\lambda\ge\lambda_0$, $n\in\Z_+$, $t\in[0,T]$ we have
\begin{equation}\label{itoapr}
u_n^\lambda(X_t^n)=u_n^\lambda(x)+\int_0^t \int_\R (u_n^\lambda(X^n_{s-}+r)-u_n^\lambda(X^n_{s-}))\wt N(ds,dr)+\lambda \int_0^t u_n^\lambda(X^n_s)\,ds-A^n_t.
\end{equation}
\item[{\rm(ii)}] Further, for any $\eps>0$ there exists a constant $C>0$ such that for any $n\in\Z_+$, $\delta\le 1/\lambda_0$, and stopping time $\tau\in[0,T]$ we have
\begin{equation}\label{momentinequalityA}
\E |A^n_{\tau+\delta}-A^n_{\tau}|^2\le C\delta^{2(1+\frac\beta\alpha-\eps)} \|b\|_\beta^2(\|b\|_\beta+1)^2.
\end{equation}
\item[{\rm(iii)}] There exists a constant $C>0$ such that for any $n\in\Z_+$, $t\in[0,T]$ we have
\begin{equation}\label{momentinequalityonepointA}
\E |A^n_{t}|^2\le C \|b\|_\beta^2(\|b\|_\beta+1)^2.
\end{equation}
\item[{\rm(iv)}] Finally, the sequence $\{(X^n,A^n)\}_{n\in\Z_+}$ is tight in $\D_{\R^2}[0,T]$.
\end{itemize}
\end{Lemma}
\begin{proof}
(i). Since $u_n^\lambda\in\C^\infty_b$, identity \eqref{itoapr} follows immediately by an application of It\^{o}'s formula (see, e.g., \cite[Theorem 4.4.7]{Apple}) to the process $X^n$ and the function $u_n^\lambda$.

(ii). Fix $\delta\le 1/\lambda_0$, stopping time $\tau$, and $\eps>0$ small enough. For $\lambda\ge\lambda_0$, $n\in\Z_+$ we denote
\begin{equation*}
I^{n,\lambda}_t:=\int_0^t \int_\R (u_n^\lambda(X^n_{s-}+r)-u_n^\lambda(X^n_{s-}))\wt N(ds,dr),\quad t\in[0,T].
\end{equation*}

It follows from \eqref{itoapr} and Proposition \ref{p:resolvent}(i) that for any $\lambda\ge\lambda_0$ we have
\begin{align}\label{pervijshag_otzenka}
|A^n_{\tau+\delta}-A^n_{\tau}|&\le |u_n^\lambda(X_{\tau+\delta}^n)-u_n^\lambda(X_{\tau}^n)|+|I^{n,\lambda}_{\tau+\delta}-I^{n,\lambda}_{\tau}|
+\lambda \int_{\tau}^{\tau+\delta} |u_n^\lambda(X^n_s)|\,ds\nn\\
&\le \snorm{u_n^\lambda}(2+\lambda\delta)+ |I^{n,\lambda}_{\tau+\delta}-I^{n,\lambda}_{\tau}|.
\end{align}
Now let us pick $\lambda:=\delta^{-1}$. Since $\delta\le 1/\lambda_0$, we clearly have $\lambda\ge \lambda_0$. Then it follows from Proposition \ref{p:resolvent}(i) and the bound $\|b_n\|_\beta\le 2 \|b\|_\beta$ that
\begin{equation}\label{Amoment1}
\snorm{u_n^\lambda}(2+\lambda\delta)\le C\delta^{1+\frac\beta\alpha-\eps} \|b\|_\beta(\|b\|_\beta+1),
\end{equation}
where the constant $C>0$ depends only on $\alpha$, $\beta$ and $\eps$ (but not $n$, $\delta$, $\lambda$, or $\tau$).

Note that for any $\gamma\in(0,1)$, we have
\begin{equation*}
|u_n^\lambda(X^n_{s-}+r)-u_n^\lambda(X^n_{s-})|\le \|u_n^\lambda\|_\gamma(1\wedge |r|^\gamma),\quad s\in[0,T],\,r\in\R.
\end{equation*}
Thus, we can take $\gamma:=\alpha/2+\alpha\eps$ and deduce from Lemma~\ref{L:vtorojmoment} and Proposition \ref{p:resolvent}(i) that
\begin{equation*}
\E (I^{n,\lambda}_{\tau+\delta}-I^{n,\lambda}_{\tau})^2\le C \delta (\|u_n^\lambda\|_{\alpha/2+\alpha\eps})^2\le
C\delta^{2(1+\frac\beta\alpha-\eps)} \|b\|_\beta^2(\|b\|_\beta+1)^2,
\end{equation*}
where again the constant $C$ does not depend on $n$, $\delta$, $\lambda$, or $\tau$. Combining this bound with \eqref{pervijshag_otzenka} and \eqref{Amoment1} we establish \eqref{momentinequalityA}.

(iii). It is clear that for any $t\in[0,T]$ there exists $N\in\Z_+$ and a increasing sequence $(t_i)_{i\in[0,N]}$ such that $t_0=0$, $t_N=t$ and $t_{i+1}-t_i\le 1/\lambda_0$. Further, one can take $N=\lceil T\lambda_0 \rceil$. Then it follows from part (ii) of the lemma and the fact that $A_0=0$ that
\begin{equation*}
\E |A^n_{t}|^2\le N \sum_{i=0}^{N-1}\E|A^n_{t_{i+1}}-A^n_{t_{i}}|^2\le C (T,\lambda_0,\alpha,\beta) \|b\|_\beta^2(\|b\|_\beta+1)^2,
\end{equation*}
which proves \eqref{momentinequalityonepointA}.

(iv). To establish the tightness of $\{(X^n,A^n)\}_{n\in\Z_+}$, first let us verify that the sequence $(A^n)_{n\in\Z_+}$ is tight in $\D_\R[0,T]$.
We would like to apply the Aldous theorem \cite[Theorem~1]{Ald}. Thus, we need to check that for each $t\in[0,T]$ the sequence of random variables $(A^n_t)_{n\in\Z_+}$ is tight; and that for any sequence of stopping times $(\tau_n)_{n \in \Z_+}$ and constants
$\delta_n\to0$ we have
\begin{equation}\label{eq:t3}
A^n_{\tau_n+\delta_n} - A^n_{\tau_n}\to0,\,\text{ as $n\to\infty$ in probability}.
\end{equation}

The first condition of Aldous' theorem holds thanks to part (iii) of the lemma. Indeed, bound
\eqref{momentinequalityonepointA} yields that for each fixed $t\in[0,T]$ the sequence $(A^n_t)_{n\in\Z_+}$ is tight.

To verify the second condition of Aldous' theorem we take a sequence of stopping times $(\tau_n)_{n \in\Z_+}$ and a sequence of constants $\delta_n\to0$. We can assume without loss of generality that $\delta_n\le 1/\lambda_0$ for all $n\in\Z_+$. Then we apply part (ii) of the lemma with $\tau=\tau_n$, $\delta=\delta_n$.  We derive
\begin{equation}\label{prov2}
\E |A^n_{\tau_n+\delta_n}-A^n_{\tau_n}|\le C \|b\|_\beta (\|b\|_\beta+1) \sqrt \delta_n\to 0,\,\text{ as $n\to\infty$},
\end{equation}
where we used the fact that $1+\beta/\alpha>1/2$. This implies \eqref{eq:t3}. Thus, all the conditions of \cite[Theorem~1]{Ald} are satisfied and the sequence $(A^n)_{n\in\Z_+}$ is tight.

Recall that $X^n_t=A^n_t+x+L_t$. It follows from \eqref{momentinequalityonepointA} and \eqref{prov2} that the sequence $(X^n)_{n\in\Z_+}$ also satisfies the conditions of \cite[Theorem~1]{Ald}. Hence, $(X^n)_{n\in\Z_+}$ is tight.

To complete the proof it remains to note that $A^n$ is continuous in $t$ for each $n$. Thus, $(A^n)_{n\in\Z_+}$ is actually $C$-tight and therefore by \cite[Corollary~VI.3.33(b)]{JS87} the sequence $\{(X^n,A^n)\}_{n\in\Z_+}$ is tight in $\D_{\R^2}[0,T]$.
\end{proof}

Now we are ready to prove the main result of this subsection.

\begin{proof}[Proof of Proposition \ref{p:wex}(i)]
Fix $\lambda\ge\lambda_0(\beta,2\|b\|_\beta)$. In the proof for brevity we will write $u:=u^\lambda_{b,b}$ and $u_n:=u^\lambda_{b_n,b_n}$.

 We use the approximating sequence $\{(X^n,A^n)\}_{n\in\Z_+}$ constructed in Lemma~\ref{L:apseq}. It follows from the lemma that this sequence is tight in $\D_{\R^2}[0,T]$. Hence by the Prokhorov
theorem there exists a subsequence $(n_k)$ such that $(X^{n_k},A^{n_k})$ converges weakly in the Skorokhod space $\D_{R^2}[0,T]$ to the limit $(X,A)$. In order not to overburden the notation,
we suppose that we have already passed to this subsequence and thus we assume that $(X^n,A^n)$ converges weakly to $(X,A)$. Then by the Skorokhod representation theorem (see, e.g., \cite[Theorem~3.1.8]{EK86}) there exists a sequence of random elements $(\wh X^n,\wh A^n)$
defined on a common probability space
$(\wh{\Omega},\wh{\F},\wh{\P})$ such that $(\wh X^n,\wh A^n)\to (\wh{X},\wh{A})$ a.s. in the Skorokhod metric and $\Law(\wh X^n,\wh A^n)=\Law({X^n},{A^n})$, $\Law(\wh{X},\wh{A})=\Law({X},{A})$.

Denote $\wh L^n:=\wh X^n-\wh A^n-x$, and define similarly $\wh L$. By above, $\Law(\wh L^n)=\Law(X^n-A^n-x)=\Law(L)$. Thus, $\wh L^n$ is an $\alpha$-stable L\'{e}vy process.
It follows from Lemma~\ref{P:PK}, that $(\wh X^n,\wh A^n,\wh L^n)$ converges a.s. in the Skorokhod metric to $(\wh{X},\wh{A},\wh{L})$. Hence $\wh{L}$ is also an $\alpha$-stable L\'{e}vy process. Denote by $\wt{\wh N^n}$ (respectively $\wt{\wh N}$) the compensated Poisson random measure of $\wh L^n$ (respectively $\wh L$).

It follows from the above considerations and \eqref{itoapr} that for $t\in[0,T]$
\begin{equation}\label{itoaprnew}
u_n(\wh X_t^n)-u_n(x)-\lambda \int_0^t u_n(\wh X^n_s)\,ds+\wh A^n_t=\int_0^t \int_\R (u_n(\wh X^n_{s-}+r)-u_n(\wh X^n_{s-}))\wt{\wh N^n}(ds,dr).
\end{equation}

Let us pass to the limit as $n\to\infty$ in \eqref{itoaprnew}.

First we recall that $\wh X^n$ converges a.s. to $\wh X$ in the Skorokhod metric as $n\to\infty$. By Proposition \ref{p:resolvent}(i)
we have $\sup_n\|u_n\|_{(1+\alpha)/2}<\infty$ and by Proposition \ref{p:resolvent}(ii) $\lim_{n\to\infty }\|u_n-u\|_{(1+\alpha)/2}\to 0$. Therefore, by Lemma~\ref{P:PK},
\begin{equation}\label{perv}
u_n(\wh X_\cdot^n)\to u(\wh X_\cdot),\,\,\text{as $n\to\infty$ a.s. in $\D_\R[0,T]$.}
\end{equation}
By Lemma~\ref{L:425},
\begin{equation}\label{vtor}
\lambda \int_0^\cdot u_n(\wh X^n_s)\,ds\to\int_0^\cdot u(\wh X_s)\,ds\,\,\text{as $n\to\infty$ a.s. in $\D_\R[0,T]$.}
\end{equation}
Note that the function $\int_0^\cdot u(\wh X_s)\,ds$ is continuous; recall that $\wh A^n$ converges a.s. to a continuous function $\wh A$. Therefore \eqref{perv}, \eqref{vtor} and \cite[Proposition~VI.1.23]{JS87} yield that the left--hand side of \eqref{itoaprnew} converges a.s. in $\D[0,T]$ to
\begin{equation*}
u(\wh X_\cdot)-u(x)-\lambda \int_0^\cdot u(\wh X_s)\,ds+\wh A.
\end{equation*}

Recall that $(\wh X^n,\wh L^n)$ converges a.s. in $\D_{\R^2}[0,T]$ to $(\wh{X},\wh{L})$ and $\sup_n\|u_n\|_{(1+\gamma)/2}<\infty$. Thus, all the conditions of Lemma~\ref{L:42} are satisfied. Hence the right--hand side of \eqref{itoaprnew} converges in probability in $\D[0,T]$ to
\begin{equation*}
\int_0^\cdot \int_\R (u(\wh X_{s-}+r)-u(\wh X_{s-}))\wt{\wh N}(ds,dr).
\end{equation*}
Thus,
\begin{equation}\label{vazhno}
u(\wh X_t)-u(x)-\lambda \int_0^t u(\wh X_s)\,ds+\wh A_t=\int_0^t \int_\R (u(\wh X_{s-}+r)-u(\wh X_{s-}))\wt{\wh N}(ds,dr),\quad t\in[0,T].
\end{equation}
Since $\wh A_t=\wh X_t-\wh L_t-x$, we see that $(\wh X, \wh L)$ is indeed a weak solution to \eqref{usde}.
\end{proof}
\begin{proof}[Proof of Proposition \ref{p:wex}(ii)] Recall that by definition
\begin{equation*}
\wh X_t=x+\wh A_t +\wh L_t,\quad t\in[0,T].
\end{equation*}
Thus it remains to check that the process $\wh A_t$ satisfies the second and third properties in the Definition~\ref{D:sol}.

To check the second property take any approximating sequence $(b_n)_{n\in\Z_+}\in\C^\infty_b$ such
that $b_n\to b$ in $\C^\beta$ as $n\to\infty$ and $\|b_n\|_\beta\le 2 \|b\|_\beta$, $n\in\Z_+$. Take any $\lambda\ge\lambda_0(\beta,2\|b\|_\beta)$. For any $n,m\in\Z_+$ we consider $u^{n,m}:=u^\lambda_{b_n,b_m}$, which is the unique $\C^{(1+\alpha/2)}$ solution to the equation
\eqref{me} with $b_n$ in place of $f$ and $b_m$ in place of $g$.
We apply It\^{o}'s formula to the process $\wh X^n$. We get
\begin{multline}\label{itomn}
u^{n,m}(\wh X_t^n)-u^{n,m}(x)-\lambda \int_0^t u^{n,m}(\wh X^n_s)\,ds+\int_0^t b_m(\wh X^n_s)\,ds\\
=\int_0^t \int_\R (u^{n,m}(\wh X^n_{s-}+r)-u^{n,m}(\wh X^n_{s-}))\wt{\wh N^n}(ds,dr).
\end{multline}

Consider now $u^{(m)}:=u^\lambda_{b,b_m}$. Then by Proposition \ref{p:resolvent}(ii)
\begin{equation*}
\lim_{n\to\infty}\|u^{n,m}-u^{(m)}\|_{(1+\alpha)/2}\to0.
\end{equation*}

Now for each fixed $m\in\Z_+$ we pass to the limit as $n\to\infty$ in \eqref{itomn}. Arguing exactly as in in part (i) of the proof, we apply Lemmas~\ref{P:PK}, \ref{L:425}, \ref{L:42} and \cite[Proposition~VI.1.23]{JS87}, to obtain
\begin{multline*}
u^{(m)}(\wh X_t)-u^{(m)}(x)-\lambda \int_0^t u^{(m)}(\wh X_s)\,ds+\int_0^t b_m(\wh X_s)\,ds\\
=\int_0^t \int_\R (u^{(m)}(\wh X_{s-}+r)-u^{(m)}(\wh X_{s-}))\wt{\wh N}(ds,dr).
\end{multline*}

Comparing this identity with \eqref{vazhno}, we deduce
\begin{equation}\label{jim}
\Bigl\|\int_0^\cdot b_m(\wh X_s)\,ds-\wh A\Bigr\|\le \snorm{u^{(m)}-u}\bigl(2+(C+\lambda) T+2\sum_{s\le T}\I(|\Delta L_s|>1)\bigr)
+\snorm{J^m},
\end{equation}
where we denoted
\begin{equation*}
J^m(t):=\int_0^t \int_{|r|<1} \bigl(u^{(m)}(\wh X_{s-}+r)-u^{(m)}(\wh X_{s-})
-u(\wh X_{s-}+r)-u(\wh X_{s-})\bigr)\wt{\wh N}(ds,dr),\quad t\in[0,T].
\end{equation*}

Clearly, for any $x\in\R$, $y\ge0$
\begin{equation*}
|u^{(m)}(x+y)-u^{(m)}(x)-u(x+y)-u(x)|\le \int_{x}^{x+y} |(u^{(m)})'(s)-u'(s)|\, ds\le y\|u^{(m)}-u\|_{(1+\alpha)/2}.
\end{equation*}

Taking into account this inequality and the fact that $J^m$ is a martingale, we apply Doob's inequality to derive for any $\eps>0$
\begin{equation}\label{jimdva}
\P(\snorm{J^m}>\eps)\le \eps^{-2}\E J^m(T)^2\le C T\eps^{-2}\|u^{(m)}-u\|^2_{(1+\alpha)/2}.
\end{equation}
Recall that $\|u^{(m)}-u\|\to 0$ as $m\to\infty$. Hence, combining \eqref{jim} and \eqref{jimdva}, we get
\begin{equation*}
\Bigl\|\int_0^\cdot b_m(\wh X_s)\,ds-\wh A\Bigr\|\to 0,\text{ in probability as $m\to\infty$}.
\end{equation*}

It now remains to show that $\wh{A}$ satisfies the third condition. Fix any $0\le s \le t \le T$. By the standard argument, we see that it is enough to check \eqref{3cond} only
for $s$, $t$ close enough. Thus, we can assume that $|t-s|\le \frac1{\lambda_0(\beta,2\|b\|_\beta)}$. It follows from Lemma~\ref{L:apseq}(ii) that
\begin{equation*}
\E |\wh A^n_{t}-\wh A^n_{s}|^2=\E |A^n_{t}- A^n_{s}|^2\le C|t-s|^{2(1+\frac\beta\alpha-\eps)} \|b\|_\beta^2(\|b\|_\beta+1)^2.
\end{equation*}
By Fatou's lemma,
\begin{equation*}
\E |\wh A_{t}-\wh A_{s}|^2\le C|t-s|^{2(1+\frac\beta\alpha-\eps)} \|b\|_\beta^2(\|b\|_\beta+1)^2.
\end{equation*}
This implies that $\wh{A}$ satisfies the third condition. Hence $\wh X$ is a weak solution to \eqref{mainsde}. This concludes the proof.
\end{proof}

\subsection{Proof of Proposition \ref{p:pu}: pathwise uniqueness}\label{SS:PU}
The primary purpose of this section is to prove Proposition \ref{p:pu}. We need the following auxiliary lemma.

\begin{Lemma}\label{L:45}
Let $\gamma_1, \gamma_2\in[0,1)$. Denote $\gamma:=\gamma_1+\gamma_2$ and assume that $\gamma\neq1$. Then for any
$f\in\C^{\gamma}$, and $r, x_1,x_2\in \R$  we have
\begin{equation}\label{est2holder}
|f(x_1+r)-f(x_1)-f(x_2+r)+f(x_2)|\le 2 \|f\|_{\gamma}|x_1-x_2|^{\gamma_1} |r|^{\gamma_2}.
\end{equation}
\end{Lemma}
The proof of the lemma is standard. For the sake of completeness we provide it in the appendix.

\begin{proof}[Proof of Proposition \ref{p:pu}]
First of all we note that thanks to Proposition \ref{p:resolvent}(i), there exists $\lambda_1=\lambda_1(\beta, \|b\|_\beta)\ge\lambda_0(\beta, \|b\|_\beta)$ such that for any $\lambda>\lambda_1$ we have
\begin{equation}\label{smallderivative}
\snorm{(u^\lambda_{b,b})'}+\snorm{u^\lambda_{b,b}}\le1/4.
\end{equation}
From now on we fix $\lambda>\lambda_1$ and for brevity write $u$ instead of $u^\lambda_{b,b}$.

It is straightforward to see that \eqref{smallderivative} yields for any $x_1,x_2\in\R$
\begin{equation}\label{vmseto_byvshih_lemm}
|u(x_1)+x_1-u(x_2)-x_2|\ge |x_1-x_2|-|u(x_1)-u(x_2)|\ge|x_1-x_2| (1-\snorm{u'})\ge\frac34|x_1-x_2|.
\end{equation}
Similarly, for any $x_1,x_2\in\R$
\begin{equation*}
|u(x_1)+x_1-u(x_2)-x_2|\le \frac54|x_1-x_2|.
\end{equation*}

Denote
\begin{equation*}
h(x_1,x_2,r):= u(x_1+r)-u(x_1)-u(x_2+r)+u(x_2),\quad x_1,x_2,r\in\R,
\end{equation*}
We have for $x_1,x_2,r\in\R$
\begin{eqnarray}\label{hest}
|h(x_1,x_2,r)|\le 2\snorm{u'}(|x_1-x_2|\wedge|r|)
&\le&\frac12(|x_1-x_2|\wedge|r|)\\
\label{hest22}
&\leq&\frac{2}{3} |u(x_1)+x_1-u(x_2)-x_2|,
\end{eqnarray}
where the last inequality follows by~\eqref{vmseto_byvshih_lemm} .

It is also clear that
\begin{equation}\label{hest2}
|h(x_1,x_2,r)|\le 4\snorm{u}\le1,\quad x_1,x_2,r\in\R.
\end{equation}

Now let introduce a sequence of functions $(V_n)_{n\in\N}$, which are smooth approximations of the function $x\mapsto|x|$. For $n\in\N$ we put
\begin{equation*}
V_n(x)=\left\{
     \begin{array}{ll}
      \frac3{8n}+\frac34nx^2-\frac18n^3x^4, & \hbox{if $|x|\le\frac1n$;} \\
      |x|, & \hbox{if $|x|\ge\frac1n$.}
     \end{array}
    \right.
\end{equation*}
We see that each $V_n$ is a twice continuously differentiable function. Further, we have $\snorm{V_n'}\le3/2$; $V_n''(x)=0$ if $|x|\ge1/n$ and $V_n''(x)\le 3n/2$ if $|x|\le 1/n$.

Now let $X^1$, $X^2$ be two weak solutions to \eqref{usde} starting at $x\in\R$. Our goal is to show that for any $t\ge0$ we have $X^1(t)=X^2(t)$ a.s.

Put $W^i_t = u(X^i_t)+X^i_t$, $i=1,2$. We apply It\^{o}'s formula to get for each $n\in\N$
\begin{align}\label{eq:16}
V_n(W^1_t-W^2_t)=& \lambda \int_0^t V_n'(W^1_s-W^2_s) (u(X^1_s)-u(X^2_s))\,ds \nn\\
&+\int_0^t\int_{\R}[V_n\bigl(W^1_s-W^2_s+h(X^1_s,X^2_s,r)\bigr)-V_n(W^1_s-W^2_s)]\,\wt N(ds,dr)\nn\\
&+\int_0^t\int_{|r|>1}\psi_n(W^1_s-W^2_s,h(X^1_s,X^2_s,r))\,\nu(dr)ds\nn\\
&+\int_0^t\int_{|r|\le1}\psi_n(W^1_s-W^2_s,h(X^1_s,X^2_s,r))\,\nu(dr)ds\nn+ V_n(0)\\
=:&I_1+I_2+I_3+I_4+\frac3{8n},
\end{align}
where we used the fact that $W^1_0-W^2_0=0$ and denoted
\begin{equation*}
\psi_n(x,y):=V_n(x+y)-V_n(x)-yV_n' (x),\quad x,y\in\R.
\end{equation*}

Let us estimate $\E |X^1_t-X^2_t|$. To do it we bound the expected values of all the four terms in the right-hand side of \eqref{eq:16}. It is easy to deal with the first term.
\begin{equation}\label{venediktov}
\E I_1\le \lambda \snorm{V_n'}\int_0^t\E |u(X^1_s)-u(X^2_s)|\,ds\le \lambda \int_0^t \E|X^1_s-X^2_s|\,ds,
\end{equation}
where we used the fact that $\snorm{u'}\le1/4$.

Since, by \eqref{hest} and \eqref{hest2}
\begin{equation*}
|V_n\bigl(W^1_s-W^2_s+h(X^1_s,X^2_s,r)\bigr)-V_n(W^1_s-W^2_s)|\le \snorm{V_n'}\,\,|h(X^1_s,X^2_s,r)|\le |r|\wedge1,
\end{equation*}
we see that $I_2$ is a martingale and hence
\begin{equation}\label{dva}
\E I_2=0.
\end{equation}

Now let us pass to the analysis of $I_3$. Using again the fact that $\snorm{V_n'}\le 3/2$, we get
\begin{equation*}
|\psi_n(x,y)|\le 2\snorm{V_n'}|y|\le 3|y|,\quad x,y\in\R.
\end{equation*}
Therefore, taking into account \eqref{hest}, we derive
\begin{equation}\label{tri}
\E I_3\le 3\int_0^t\int_{|r|>1}\E|h(X^1_s,X^2_s,r)|\,\nu(dr)ds \le C\int_0^t \E|X^1_s-X^2_s|ds.
\end{equation}

Finally, to analyze $I_4$ we recall that \eqref{hest22} implies for any $s\in[0,T]$
\begin{equation*}
|h(X^1_s,X^2_s,r)|\leq \frac23 |W^1_s-W^2_s|, \quad r\in \R.
\end{equation*}
Therefore if $W^1_s-W^2_s\ge 3/n$, we immediately get
\begin{equation*}
W^1_s-W^2_s+ h(X^1_s,X^2_s,r)\geq \frac1n, \quad r\in \R,
\end{equation*}
and if $W^1_s-W^2_s\le -3/n$, we have
\begin{equation*}
W^1_s-W^2_s+ h(X^1_s,X^2_s,r)\leq -\frac1n, \quad r\in \R.
\end{equation*}
Thus, it follows from the definitions of $\psi_n$ and $V_n$, that if $|W^1_s-W^2_s|\ge 3/n$, then
\begin{equation*}
\psi_n(W^1_s-W^2_s,h(X^1_s,X^2_s,r))=0,\quad r\in \R.
\end{equation*}
Therefore we have,
\begin{equation}\label{I4step1}
I_4= \int_0^t\int_{|r|\le1}\I(|W^1_s-W^2_s|\le 3/n) \psi_n(W^1_s-W^2_s,h(X^1_s,X^2_s,r))\,\nu(dr)ds.
\end{equation}
Note also that $|V_n''(x)|\le\frac32n\I(|x|\le1/n)$, $x\in\R$ and thus
\begin{equation}
\label{1_1_a}
|\psi_n(x,y)|\le \frac{3n}{2}|y|^2,\quad x,y\in\R.
\end{equation}

We take small $\eps>0$ such that $\alpha/2+\eps<1$ and $1/2+\alpha/2+\eps<\alpha+\beta$.
Then we use \eqref{I4step1}, \eqref{1_1_a} and apply  Lemma~\ref{L:45} with $\gamma_1=1/2+\eps/2$, $\gamma_2=\alpha/2+\eps/2$ to derive
\begin{align}\label{arba}
\E I_4&\le C n \int_0^t\int_{|r|\le1}\E\I(|W^1_s-W^2_s|\le 3/n) |h(X^1_s,X^2_s,r)|^2 \,\nu(dr) ds\nn\\
&\le C n \|u\|^2_{\frac{\alpha+1}{2}+\eps}
\int_0^t\int_{|r|\le1}\E\I(|W^1_s-W^2_s|\le 3/n) |W^1_s-W^2_s|^{2(1/2+\eps/2)}r^{2(\alpha/2+\eps/2)}\,\nu(dr)ds\nn\\
&\le Ct n^{-\eps}.
\end{align}

Combining now \eqref{eq:16}, \eqref{venediktov}, \eqref{dva}, \eqref{tri}, \eqref{arba}, we finally deduce
\begin{equation*}
\E V_n(W_t^1-W_t^2)\le \frac3{8n}+C(\lambda)\int_0^t\E|X^1_s-X^2_s|ds+Ct n^{-\eps}.
\end{equation*}
By passing to the limit as $n\to\infty$ and taking into account \eqref{vmseto_byvshih_lemm} we obtain
\begin{equation*}
\E |X_t^1-X_t^2|\le 2\E |W_t^1-W_t^2|\le 2C(\lambda)\int_0^t\E|X^1_s-X^2_s|ds.
\end{equation*}
Clearly, $\E|X^i_s|<\infty$, $s\in[0,T]$, $i=1,2$. Thus, by Gronwall's inequality
\begin{equation*}
\E|X^1_t-X^2_t|=0.
\end{equation*}
This implies the statement of the lemma.
\end{proof}

\subsection{Proof of Theorem \ref{T:1}}\label{sec:kur}

We have already shown that (\ref{mainsde}) has a weak solution (Proposition \ref{p:wex}(ii)). By Proposition \ref{p:pu} we know that pathwise uniqueness
holds for (\ref{usde}) and via Proposition \ref{p:vtreal} we know that
every weak solution to (\ref{mainsde}) is a weak solution to
(\ref{usde}). Therefore we have shown pathwise uniqueness for
(\ref{mainsde}). Now we can apply a generalized version of the classical Yamada-Watanabe theorem, see 
\cite[Theorem 3.4 and Proposition~2.13]{K14}. The strong existence and pathwise
uniqueness for (\ref{mainsde}) then imply strong uniqueness for the same equation.
For the sake of completeness, we rewrite our equation (\ref{mainsde}) in the notation
of \cite{K14} and verify that all the assumptions required for Theorem 3.4 and Proposition~2.13
in \cite{K14} are satisfied.

Let $x \in \R$. Let $(\Omega, {\cal F}, \P)$ be a complete probability space and $(L_t)_{t \geq  0}$ be the symmetric stable $\alpha$-process on it. Let $\Upsilon$ be the product measure of $\delta_x$ and the law of $L$. Let $S_1 = \D_\R[0,\infty)$ and $S_2 = \R \times
 \D_\R[0,\infty)$. Let $ {\cal P}(S_1 \times
  S_2)$ be the space of probability measures on $S_1 \times
  S_2$ with the product Borel $\sigma$-algebras of $S_1$ and $S_2$. Let $X$ be a random variable on $(\Omega, {\cal F}, \P)$ taking values in $S_1,$ $Y$ be a random variable $(\Omega, {\cal F}, \P)$ taking values in $S_2$ and
  $\mu_{X \times Y} \in {\cal P}(S_1 \times S_2)$ be the joint
  distribution of $(X,Y)$.  Moreover we assume that $Y=(x, (L_s)_{s\geq 0})$, that is,
the law of $Y$ is $\Upsilon$.  Then our model (\ref{mainsde}) is specified by a set of constraints
  $\Gamma$ relating $(X,Y)$. More precisely, we denote a solution to
  equation (\ref{mainsde}) as $X$ that satisfy
  \begin{equation*}
\Gamma(X,Y) =0,
\end{equation*}
where $\Gamma $ is given by the constraints of Definition \ref{D:sol}. We denote by $$S_{\Gamma, \Upsilon} : = \{ \mu_{(X,Y)} \in {\cal P}(S_1 \times S_2) : \Gamma(X,Y) = 0\ \mbox{ and } \mu_{X,Y}(S_1 \times \cdot) = \Upsilon(\cdot) \}.$$


We will follow \cite{K14} in defining the notion of compatible solution.

\begin{Definition}[\textbf{Compatible Solutions}] \label{cad}
For each
$t\geq 0$,
let
$\{{\mathcal F}^X_t\}$ and $\{{\mathcal F}^Y_t\}$ be complete filtrations generated by $X$ and $Y$ respectively (see Remark~2.3 in~\cite{K14} for the precise definition of completion). The collection
$${\mathcal C} \equiv \{ ({\mathcal F}^X_t, {\mathcal F}^Y_t) : t \geq 0\}$$
will be referred to as a compatibility structure.
  $X$ is said to be ${\mathcal C}$-compatible with $Y$ if for each $t \geq 0$, and $h \in \SL^1(S_2, \Upsilon)$
  $$ \E(h(Y) \mid {\mathcal F}^X_t \vee {\mathcal F}^Y_t ) = \E(h(Y) \mid {\mathcal F}^Y_t ).$$
  Finally let $$S_{\Gamma, {\mathcal C}, \Upsilon} : = \{ \mu_{(X,Y)} \in S_{\Gamma, \Upsilon} : X \mbox{ is } {\mathcal C}-\mbox{compatible with } Y \}.$$
\end{Definition}

\begin{Lemma} $(X,{L})$ is a weak solution of (\ref{mainsde}) if and only if $ \mu_{(X,Y)} \in S_{\Gamma, {\mathcal C}, \Upsilon}.$ \label{ceqad}
\end{Lemma}
\begin{proof}
 Let $(X,L)$ be a weak solution to (\ref{mainsde}) adapted to a complete filtration ${\cal F}_t.$
Define $Y_{\cdot\leq t} := (x, (L_{\min{\{s, t\}}})_{s \geq 0})$ and
$Y^{\cdot\geq t} := (L_{t +s} -L_t)_{s \geq 0}$.

Clearly, for each $t \geq 0$ we have ${\cal F}^Y_t \subset {\cal F}^X_t \vee {\cal F}^Y_t \subset \mathcal{F}_t$. Note that  $Y_{\cdot\leq t} $ is $ {\cal F}^Y_t$-measurable. Further, 
$(L_t)_{t\geq 0}$ is an $\alpha$-stable process with respect to the filtration $\mathcal{F}_t$ and thus also with respect to
${\cal F}^X_t \vee {\cal F}^Y_t$. This implies that
 $Y^{\cdot\geq t}$  is independent of $ {\cal F}^X_t \vee {\cal F}^Y_t$. For any $h \in \SL^1(S_2, \Upsilon)$ and for all $t \geq 0$, there exist bounded measurable functions $h_t$ on ${\R\times\D_\R[0,\infty) \times \D_\R[0,\infty)}$ such $h(Y) = h_t(Y_{\cdot\leq t},Y^{\cdot\geq t})$ a.s. Then, {following the argument in the proof of Lemma~2.4 in~\cite{K14}, we get}
\begin{align*}
 \E(h(Y) \mid {\cal F}^{X}_t \vee  {\cal F}^Y_t) =& \E(h_t(Y_{\cdot\leq t}, Y^{\cdot\geq t}) \mid {\cal F}^{X}_t \vee {\cal F}^Y_t) \\
 =& \E\Bigl( \int h_t(Y_{\cdot\leq t},y) \P(Y^{\cdot\geq t} \in dy) \Bigm| F^{X}_t \vee {\cal F}^Y_t \Bigr)\\
=& \int h_t(Y_{\cdot\leq t},y) \P(Y^{\cdot\geq t}\in dy)\\
= & \E\Bigl( \int h_t(Y_{\cdot\leq t},y) \P(Y^{\cdot\geq t} \in dy) \Bigm| {\cal F}^Y_t\Bigr)\\
=& \E(h(Y) \mid
    {\cal F}^Y_t)
\end{align*}
{Thus} $X$ is ${\mathcal C}$-compatible with $Y$ and therefore $ \mu_{(X,Y)} \in S_{\Gamma, {\mathcal C}, \Upsilon}$.

For the converse, let $(X,Y)$ be such that $ \mu_{(X,Y)} \in S_{\Gamma, {\mathcal C}, \Upsilon}.$ Take ${\mathcal F}_t = {\cal F}^{X}_t \vee {\cal F}^Y_t$. Then it follows that on the complete filtered probability space $(\Omega,{\cal F}, (\F_t)_{t\ge 0}, \P)$ the process $X_t$ is adapted to ${\cal F}_t$, the process $L_t$ is an $(\F_t)_{t\ge 0}$ adapted symmetric $\alpha$-stable process and $(X,L)$ satisfies (\ref{mainsde}). Hence $(X,L)$ 
is indeed a weak solution of (\ref{mainsde}).
\end{proof}
To complete the proof of Theorem \ref{T:1} we need the following definition.
\begin{Definition}
We say that pointwise uniqueness holds in $S_{\Gamma, {\mathcal C}, \Upsilon}$ if for any random elements $X_1$, $X_2$ and $Y$ defined on the same probability space with $ \mu_{(X_1,Y)} \in S_{\Gamma, {\mathcal C}, \Upsilon}$ and $ \mu_{(X_2,Y)} \in S_{\Gamma, {\mathcal C}, \Upsilon}$ one has $X_1 = X_2$ a.s.
\end{Definition}

\medskip

Now we are ready to finish \\
\noindent\textit{Proof of Theorem \ref{T:1}.}
By Proposition \ref{p:wex} we know {that} a weak solution exists to
(\ref{mainsde}). Thus $S_{\Gamma, \Upsilon} \neq \emptyset.$ By Proposition
\ref{p:pu} and Proposition \ref{p:vtreal} we know pathwise uniqueness {holds for (\ref{mainsde})}.
Then the converse part of Lemma~\ref{ceqad} implies that pointwise uniqueness holds in $S_{\Gamma, {\mathcal C}, \Upsilon}$. 
Then pointwise uniqueness in $S_{\Gamma, {\mathcal C}, \Upsilon}$ along with \cite[Lemma 2.10]{K14} and the direct part of Lemma~\ref{ceqad}, implies that the hypotheses of \cite[Theorem 3.4]{K14} are satisfied. Now \cite[Theorem 3.4 and Proposition 2.13]{K14} together imply that (\ref{mainsde}) has a unique strong solution. (As a caution to the reader, to avoid any confusion, we note that the word ``strong'' has a different meaning in~\cite{K14}.)
\qed

\begin{appendices}

\section{Proof of auxiliary results}
\subsection{Proof of Lemma~\ref{P:1} } \label{p32a}
\begin{proof}
(i). We use an argument similar to the one used in the proof of \cite[Theorem~2.3.8]{T}. Let $F f$ and $F^{-1}f $ denote the Fourier transform and the inverse Fourier transform
of $f\in \S'$, respectively. We have
\begin{equation}\label{TriebelStep1}
\|\OL_\alpha f \|_{\gamma-\alpha}=\|F^{-1}|x|^\alpha F f \|_{\gamma-\alpha}=\|F^{-1}|x|^\alpha(1+x^2)^{-\alpha/2} F F^{-1}(1+x^2)^{\alpha/2}Ff \|_{\gamma-\alpha}
\end{equation}

Let $W^{1}_2(\R)$ denote the Sobolev space of functions $\R\to\R$. Let $\psi\in S$ and $\phi\in S$ be a functions with
\begin{align*}
&0\le \psi(x)\le 1,\quad \supp\psi\subset \{|x|\le 4\},\quad \psi(x)=1\,\,\text{ if $|x|\le 2$}; \\
&0\le \phi(x)\le 1,\quad \supp\phi\subset \{1/4\le |x|\le 4\},\quad \phi(x)=1\,\,\text{ if $1/2\le|x|\le 2$}.
\end{align*}
Following \cite[Section~2.4.8]{T}, define for a function $m\colon \R\to\R$
\begin{equation*}
\|m\|_{h^1_2}:=\|\psi m\|_{W^1_2}+\sup_{k\in\Z_+}\|\phi(\cdot)m(2^k\cdot)\|_{W^1_2}.
\end{equation*}
We say that $m\in h^1_2$ if $\|m\|_{h^1_2}<\infty$. Then by \cite[Formula~2.6.1.(2)]{T} for any $f\in \C^{\eta}$, $\eta\in\R$, $m\in W^1_2(\R)$ we have
\begin{equation}\label{TriebelStep2}
\|F^{-1} m F f \|_{\eta}\le C\|m\|_{h^{1}_2}\|f\|_{\eta}.
\end{equation}
Let $\rho(x):=|x|^\alpha(1+x^2)^{-\alpha/2}$, $x\in\R$.
Taking in \eqref{TriebelStep2} $\eta:=\gamma-\alpha$ and combining it with \eqref{TriebelStep1}, we deduce
\begin{equation*}
\|\OL_\alpha f \|_{\gamma-\alpha}\le C \|\rho\|_{h^1_2} \| F^{-1}(1+|x|^2)^{\alpha/2}Ff \|_{\gamma-\alpha}\le
C \|\rho\|_{h^1_2} \| f \|_{\gamma},
\end{equation*}
where the second inequality follows from \cite[Theorem~2.3.8(i)]{T}. Since $\alpha>1$, we see that $\|\rho\|_{h^1_2}<\infty$. This implies the statement of the theorem.

(iii). Let $p_t$ denotes the kernel associated with the semigroup $P_t$. Then for any $t>0$, $f\in\C^\gamma$ we get
\begin{equation*}
\|P_t f\|_\gamma=\Bigl\|\int_\R p_t(y)f(\cdot-y) \, dy\Bigr\|_\gamma\le
\int_\R p_t(y)\|f(\cdot-y)\|_\gamma \, dy= \|f\|_\gamma.
\end{equation*}
This implies for any $t\ge1$
\begin{equation*}
\|P_t f\|_\gamma=\|P_{t-1}P_1 f\|_\gamma\le \|P_1 f\|_\gamma\le C \|f\|_\eta,
\end{equation*}
where the last inequality follows from Lemma~\ref{P:1}(ii).
\end{proof}

\subsection{Proof of Lemma~\ref{L:42}}\label{A:L42}
\begin{proof}
(i). The proof of this result is standard. Fix a compact set $A\subset \R\setminus \{0\}$. Also let
$D(X,L):=\{t\geq 0: \P(X_t=X_{t-} \;{\rm and}\; L_t=L_{t-})=1\}$. Note that by \cite[Lemma~3.7.7]{EK86}, the complement of $D(X,L)$ is at most countable.

For any continuous $g$ with compact support in $\R\setminus\{0\}$ we denote
$$ \wt N^n_r(g) = \int_0^r\int_{\R} g(x) \wt N^n(dx,d s),\;
\wt N_r(g) = \int_0^r\int_{\R} g(x) \wt N(dx,ds),\; \quad r\geq 0.$$
Clearly by the martingale assumption on $(\wt N^n_{r}(A))_{r\geq 0}$ and boundedness of its second moment
we easily get by the dominated
convergence theorem that
$$ \wt N^n_r(g), \; r\geq 0,$$
is also an ${\cal F}^{X^n}_r$-martingale.
Fix arbitrary $t>s\geq 0: t,s\in D(X,L)$. Chose any $m\in\Z_+$ and arbitrary $0\leq t_1<t_2<\ldots<t_m\leq s$ such that $t_1,t_2\ldots\in D(X,L)$.
 Then for any bounded continuous functions $h_1,h_2,\ldots, h_m$, and continuous $g$ with compact support in $\R\setminus\{0\}$
by the martingale property of $\wt N^n_{t}(A),\, t\geq 0$, we have
\begin{equation*}
\E \Bigl( (\wt N^n_{t}(g)-\wt N^n_{s}(g)) \prod_{i=1}^m h_i(X^n_{t_i})\Bigr)=0.
\end{equation*}
Since $s,t, t_1,t_2\ldots\in D(X,L)$, by Lemma~3.7.8
in~\cite{EK86} we have
\begin{eqnarray*}
(X^n_{t_1},\ldots, X^n_{t_m}, \wt N^n_s(g), \wt N^n_t(g))\rightarrow
(X_{t_1},\ldots, X_{t_m}, \wt N_s(g), \wt N_t(g)),
\end{eqnarray*}
a.s. as $n\rightarrow \infty$. Note also that the second moment of $\wt N^n_{t}(A) $ is bounded uniformly in $n$.
Thus by passing to the limit as
$n\rightarrow\infty$, we can use the uniform integrability to get
\begin{equation*}
\E \left( (\wt N_{t}(g)-\wt N_{s}(g)) \prod_{i=1}^m h_i(X_{t_i})\right) = \lim_{n\rightarrow \infty}
\E \left( (\wt N^n_{t}(g) -\wt N^n_{s}(g))\prod_{i=1}^m h_i(X^n_{t_i})\right)=0.
\end{equation*}
Recall that the complement of $D(X,L)$ is at most countable. Thus,
since $X$, $L$ are c\`adl\`ag, $0\leq t_1< t_2< s<t$ were arbitrary in $D(X,L)$ and $h_i, i=1,\dots,m,$ were arbitrary continuous bounded functions,
we get the desired martingale property of $\wt N_t(g), t\geq 0$ for any continuous $g$ with compact support in $\R\setminus\{0\}$ . Now use again the dominated convergence theorem to get the martingale property of
 $\wt N_t(A), t\geq 0,$ for any compact $A\in \R\setminus\{0\}$.

(ii).
First of all let us note that conditions \eqref{rkvadrat} and \eqref{cond1} and the first part of the lemma imply that $I^n$ and $I$ are well--defined; moreover these processes are square--integrable martingales (see, e.g., \cite[Theorem~4.2.3(4)]{Apple}).

Fix arbitrary $0<\eps_0<\eps$. Let $\chi$ be a smooth non--decreasing function $\R_+\to[0,1]$ such that $\chi(x)=0$ for $x\in[0,\eps_0)$; $\chi(x)=1$ for $x>\eps$; and $0\le\chi\le1$ for $x\in[\eps_0,\eps]$.

For any $n\in\Z_+$ we split $I^n$ into several parts:
\begin{align*}
I^n(t):=&\int_0^t \int_{|r|>\eps} f_n(X^n_{s-},r) N^n(ds,dr)-\int_0^t \int_{|r|>\eps} f_n(X^n_{s-},r)\,\nu(dr)ds\nn\\
&+\int_0^t \int_{|r|\le\eps} f_n(X^n_{s-},r)\wt N^n(ds,dr)\nn\\
=& \sum_{s\le t } f_n(X^n_{s-},\Delta L^n(s))\chi(|\Delta L^n(s)|)-\int_0^t \int_{|r|>\eps} f_n(X^n_{s-},r)\,\nu(dr)ds\\
&+\sum_{s\le t } f_n(X^n_{s-},\Delta L^n(s))\bigl(\I(|\Delta L^n(s)|>\eps)-\chi(|\Delta L^n(s)|)\bigr)\\
&+\int_0^t \int_{|r|\le\eps} f_n(X^n_{s-},r)\wt N^n(ds,dr)\nn\\
=:& I^n_1(t)+I^n_2(t)+I^n_3(t)+I^n_4(t),
\end{align*}
where $t\in[0,T]$. In a similar way we define the terms $I_1(t)$, $I_2(t)$, $I_3(t)$, $I_4(t)$. Let us analyze these terms. Our plan is as follows. We will show that $I^n_1$ and $I^n_2$ converges to $I_1$ and $I_2$, respectively, in the Skorokhod metric, and all the other terms: $I^n_3$, $I_3$, $I^n_4$, $I_4$ are ``small''. Since $I_2$ is continuous this would imply the required convergence of $I^n$ to $I$.

Let us implement this plan. Using the fact that the Skorokhod distance is smaller than the uniform distance, we derive
for $n\in\Z_+$
\begin{align}\label{DuhovnayaSkrepa}
d(I^n,I)\le& d(I^n,I^n_1+I^n_2)+d(I^n_1+I^n_2,I_1+I_2)+d(I_1+I_2,I)\nn\\
\le& d(I^n_1+I^n_2,I_1+I_2)+\snorm{I_3^n}+\snorm{I_3}+\snorm{I_4^n}+\snorm{I_4}.
\end{align}

To deal with $I^n_1$, we apply Lemma~\ref{P:PK} to the metric spaces $E_1:=\R^2$, $E_2:=\R$, and to the family of mappings $\Phi^n\colon \D_{E_1}[0,T]\to\D_{E_2}[0,T]$
\begin{equation*}
\Phi^n(Z_1,Z_2)(t)=\sum_{s\le t } f_n(Z_1(s-),\Delta Z_2(s))\chi(|\Delta Z_2(s)|),\quad t\in[0,T],
\end{equation*}
where $Z=(Z_1,Z_2)\in \D_{E_1}[0,T]$ and $n\in\Z_+$. We define the mapping $\Phi$ in a similar way.
We see that for any $Z\in\D_{E_1}[0,T]$, $\lambda\in\Lambda$, $n\in\Z_+$ we have $\Phi^n(Z\circ\lambda)=\Phi^n(Z)\circ\lambda$. It is also clear that if
$\snorm{Z^n-Z}\to0$, then $\snorm{\Phi^n(Z^n)-\Phi(Z)}\to0$. Thus, all the conditions of Lemma~\ref{P:PK} are met. Put now $Z^n:=(X^n,L^n)$ and recall that by assumption $(X^n,L^n)\to(X,L)$ a.s. in the Skorokhod topology. Therefore, by Lemma~\ref{P:PK} we have
\begin{equation}\label{ivan}
d(I^n_1,I_1)=d\bigl(\Phi^n(X^n,L^n),\Phi(X,L)\bigr)\to0\text{\, a.s. as $n\to\infty$}.
\end{equation}

Now we move on to $I^n_2$. It is easy to analyze this term. By Lemma~\ref{L:425}, we have
\begin{equation}\label{idva}
d(I^n_2,I_2)\le\snorm{I^n_2-I_2}\to0\text{\, a.s. as $n\to\infty$}.
\end{equation}
Since $I_2$ is continuous, it follows from \eqref{ivan}, \eqref{idva} and \cite[Proposition~VI.1.23]{JS87} that
\begin{equation}\label{ivanich}
d(I^n_1+I^n_2,I_1+I_2)\to0\text{\, a.s. as $n\to\infty$}.
\end{equation}

To study $I^n_3$ and $I_3$ we note that by definition for any $n\in\Z_+$
\begin{equation*}
|I^n_3(t)|+|I_3(t)|\le C \sum_{s\le t } \I\bigl(|\Delta L^n(s)|\in[\eps_0,\eps]\bigr),\quad t\in[0,T],
\end{equation*}
where we have also used the uniform boundedness of the sequence $(f_n)$. Therefore
\begin{equation}\label{petrovich}
\E\snorm{I^n_3}+\E\snorm{I_3}\le C T \nu([\eps_0,\eps]),\quad n\in\Z_+.
\end{equation}

Finally, let us deal with $I^n_4$ and $I_4$. It is clear that both $I^n_4$ and $I_4$ are square--integrable martingales. We apply Doob's $\SL_2$--martingale inequality (see, e.g., \cite[Theorem~2.1.5]{Apple}) to derive
\begin{align}\label{litij}
\E \snorm{I^n_4}^2+\E \snorm{I_4}^2&\le \E I^n_4(T)^2+\E I_4(T)^2\nn\\
&= \int_0^T \int_{|r|\le\eps} \E(f_n(X^n_{s-},r)^2+f(X_{s-},r)^2)\,\nu(dr)ds\nn\\
&\le CT \int_{|r|\le\eps} r^{2} \nu(dr),
\end{align}
where in the last inequality we used condition \eqref{cond1}.

Now combining \eqref{ivanich}, \eqref{petrovich}, \eqref{litij} with \eqref{DuhovnayaSkrepa}, we finally deduce
\begin{align}\label{finishiruem}
\limsup_{n\to\infty} \E[d(I^n,I)\wedge1]\le& \limsup_{n\to\infty}
\E[d(I^n_1+I^n_2,I_1+I_2)\wedge1]\nn\\
&+ \limsup_{n\to\infty}(\E\snorm{I^n_3}+\E\snorm{I_3}+\E\snorm{I^n_4}+\E\snorm{I_4})\nn\\
\le& C T \nu([\eps_0,\eps]) +C\sqrt{T} \Bigl(\int_{|r|\le\eps} r^{2} \nu(dr)\Bigr)^{1/2}.
\end{align}
Recall that $\eps>0$ and $\eps_0\in(0,\eps)$ were arbitrary. Recall that by the definition of the L\'evy measure, we have $\int_{|r|\le1} r^2\nu (dr)<\infty$. Thus, by taking consequently the limits in the right-hand side of \eqref{finishiruem} first as $\eps_0\nearrow\eps$ and then as $\eps\to0$ we get
\begin{equation*}
\limsup_{n\to\infty} \E[d(I^n,I)\wedge1]=0
\end{equation*}
and hence $d(I^n,I)$ converges to $0$ in probability.
\end{proof}

\subsection{Proof of Lemma~\ref{L:vtorojmoment}}\label{A:2mom}
\begin{proof}
Let us denote
\begin{equation*}
I_t:=\int_{0}^{t} \int_\R f(s,r,\omega)\wt N(ds,dr),\quad t\in[0,T].
\end{equation*}
As usual, to calculate the second moment of $I$ we have to treat ``small'' and ``large'' jumps separately. Therefore we split $I$ into a sum of two integrals (one with ``small'' jumps and the other with ``big'' jumps):
\begin{equation} \label{splitI}
I_t= \int_0^t\int_{|r|\ge1} f(s,r,\omega)\wt N(ds,dr)+\int_0^t\int_{|r|\le1} f(s,r,\omega)\wt N(ds,dr)
=:I^{1}_t+I^{2}_t.
\end{equation}

We begin with large jumps. Using \eqref{fbound}, we derive for any $0\le t_1\le t_2\le T$
\begin{align}\label{esti2}
|I^{1}_{t_2}-I^{1}_{t_1}|&\le \int_{t_1}^{t_2}\int_{|r|\ge1} | f(s,r,\omega)|(N(ds,dr)+\nu(r)ds)\nn\\
&\le C_f \sum_{s\in[t_1,t_2]} \I(|\Delta L_s|\ge1)+CC_f (t_2-t_1) .
\end{align}
Note that the process $t\mapsto\sum_{s\le t} \I(|\Delta L_s|\ge1)$ is a Poisson process with intensity $\nu(\R\setminus(-1;1))$ (see, e.g., \cite[Theorem 2.3.5 (1)]{Apple}). Therefore \eqref{esti2} implies for any stopping times $\tau_1,\tau_2\in[0,T]$
\begin{equation}\label{largejumps}
\E |I^{1}_{\tau_2}-I^{1}_{\tau_1}|^2\le C C_f^2\E |\tau_2-\tau_1|.
\end{equation}

To study small jumps we apply the standard machinery for calculating the second moment of a stochastic integral (see, e.g., \cite[Theorem~4.2.3(2)]{Apple}). We make use of \eqref{fbound} to get
\begin{align*}
\E(I^{2}_{\tau_2}-I^{2}_{\tau_1})^2\le&
\E \int_{\tau_1}^{\tau_2}\int_{|r|\le1} f(s,r,\omega)^2\,\nu(dr)ds\nn\\
\le&C C_f^2 \E \int_{\tau_1}^{\tau_2}\int_{|r|\le1} |r|^{2\gamma-\alpha-1}\,drds\le C C_f^2 \E |\tau_2-\tau_1|,
\end{align*}
where the second inequality holds since $\gamma>\alpha/2$ by assumption. Combining this with \eqref{largejumps} and substituting into \eqref{splitI} we obtain \eqref{secondmomentstoch}.
\end{proof}

\subsection{Proof of Lemma~\ref{L:45}}
\begin{proof}
First consider the case $\gamma\in(0,1)$. Then clearly for any $x_1,x_2,r\in\R$
\begin{align*}
&|f(x_1+r)-f(x_1)-f(x_2+r)+f(x_2)|\le 2\|f\|_{\gamma}|x_1-x_2|^\gamma;\\
&|f(x_1+r)-f(x_1)-f(x_2+r)+f(x_2)|\le 2\|f\|_{\gamma}|r|^\gamma.
\end{align*}
This immediately implies that for any $\gamma_1,\gamma_2>0$ such that $\gamma_1+\gamma_2=\gamma$ we have
\begin{equation*}
|f(x_1+r)-f(x_1)-f(x_2+r)+f(x_2)|\le 2\|f\|_{\gamma}|x_1-x_2|^{\gamma_1}|r|^{\gamma_2},
\end{equation*}
which is \eqref{est2holder}.

Now let $\gamma\in(1,2)$. Take any $x_1,x_2,r\in\R$. Then
\begin{equation}\label{gammap1}
|f(x_1+r)-f(x_1)-f(x_2+r)+f(x_2)|=\bigl|\int_{x_1}^{x_2} (f'(s)-f'(s+r))\,ds \bigr|\le \|f\|_{\gamma}|x_1-x_2||r|^{\gamma-1}.
\end{equation}
Similarly,
\begin{equation}\label{gammap2}
|f(x_1+r)-f(x_1)-f(x_2+r)+f(x_2)|\le \|f\|_{\gamma}|x_1-x_2|^{\gamma-1}|r|.
\end{equation}
Now we take any $\gamma_1,\gamma_2>0$ such that $\gamma_1+\gamma_2=\gamma$. We raise both sides of \eqref{gammap1} to the power of $(1-\gamma_2)/(2-\gamma)$, both sides of \eqref{gammap2} to the power of $(1-\gamma_1)/(2-\gamma)$, and multiply the obtained inequalities. Clearly, these powers are between $0$ and $1$. We get
\begin{equation*}
|f(x_1+r)-f(x_1)-f(x_2+r)+f(x_2)|\le \|f\|_{\gamma}|x_1-x_2|^{\gamma_1}|r|^{\gamma_2}.
\end{equation*}
This yields \eqref{est2holder}
\end{proof}

\end{appendices}


\noindent {\bf Siva Athreya}\\
8th Mile Mysore Road, Indian Statistical Institute,
     Bangalore 560059, India.\\
     Email: \texttt{athreya@isibang.ac.in}

     \vfill
\noindent {\bf Oleg Butkovsky}\\
  Technische Universit\"at Berlin, Institut f\"ur Mathematik, MA 7-5, Fakult\"at II, Strasse des 17.~Juni 136, 10623 Berlin, FRG. Email: \texttt{oleg.butkovskiy@gmail.com}

     \vfill

\noindent {\bf Leonid Mytnik}\\
Technion --- Israel Institute of Technology,
Faculty of Industrial Engineering and Management
 Haifa, 3200003, Israel.
 Email: \texttt{leonid@ie.technion.ac.il}
 \vfill


\begin{thebibliography}{10}

\bibitem{Ald}
D. Aldous.
\newblock Stopping times and tightness.
\newblock {\em Ann. Probability}, 6(2):335--340, 1978.

\bibitem{Apple}
D. Applebaum.
\newblock {\em L\'evy processes and stochastic calculus}, volume 116 of {\em
 Cambridge Studies in Advanced Mathematics}.
\newblock Cambridge University Press, Cambridge, second edition, 2009.

\bibitem{BR17}
E. Bandini and F. Russo.
\newblock Weak {D}irichlet processes with jumps.
\newblock {\em Stochastic Process. Appl.}, 127(12):4139--4189, 2017.

\bibitem{BC}
R. F. Bass and Z. Chen.
\newblock Stochastic differential equations for {D}irichlet processes.
\newblock {\em Probab. Theory Related Fields}, 121(3):422--446, 2001.

\bibitem{Bilya}
P. Billingsley.
\newblock {\em Convergence of probability measures}.
\newblock Wiley Series in Probability and Statistics: Probability and
 Statistics. John Wiley \& Sons, Inc., New York, second edition, 1999.
\newblock A Wiley-Interscience Publication.

\bibitem{BP15}
V. I. Bogachev and A.Yu. Pilipenko.
\newblock Strong solutions to stochastic equations with {L}\'evy noise and a
 nonconstant diffusion coefficient.
\newblock {\em Dokl. Akad. Nauk}, 469(5):532--534, 2016.

\bibitem{CG16}

R. Catellier, and M. Gubinelli.
\newblock Averaging along irregular curves and regularisation of ODEs.
\newblock{\em Stochastic Process. Appl.} 126(8): 2323-2366, 2016.


\bibitem{bib:csz15}
Z. Chen, R. Song, and X. Zhang.
\newblock Stochastic flows for L\'evy processes with H\"older drifts.
\newblock {\em Preprint : \href{https://arxiv.org/abs/1501.04758v1}{https://arxiv.org/abs/1501.04758v1}}, 2017.

\bibitem{bib:chenw16}
Z. Chen and L. Wang.
\newblock Uniqueness of stable processes with drift.
\newblock {\em Proc. Amer. Math. Soc.}, 144(6):2661--2675, 2016.

\bibitem{CJMS}
F. Coquet, A. Jakubowski, J. M\'emin, and L. Slomi\'nski.
\newblock Natural decomposition of processes and weak {D}irichlet processes.
\newblock In {\em In memoriam {P}aul-{A}ndr\'e {M}eyer: {S}\'eminaire de
 {P}robabilit\'es {XXXIX}}, volume 1874 of {\em Lecture Notes in Math.}, pages
 81--116. Springer, Berlin, 2006.

\bibitem{davie}
A.M. Davie.
\newblock Uniqueness of solutions of stochastic differential equations.
\newblock {\em Int. Math. Res. Not. IMRN}, (24):Art. ID rnm124, 26, 2007.

\bibitem{EK86}
S. N. Ethier and T. G. Kurtz.
\newblock {\em Markov processes}.
\newblock Wiley Series in Probability and Mathematical Statistics: Probability
 and Mathematical Statistics. John Wiley \& Sons, Inc., New York, 1986.
\newblock Characterization and convergence.

\bibitem{FGP}
F.Flandoli, M.Gubinelli, and E.Priola.
\newblock Well-posedness of the transport equation by stochastic perturbation.
\newblock {\em Invent. Math.}, 180(1):1--53, 2010.

\bibitem{bib:fir17}
F. Flandoli, E. Issoglio, and F. Russo.
\newblock Multidimensional stochastic differential equations with
 distributional drift.
\newblock {\em Trans. Amer. Math. Soc.}, 369(3):1665--1688, 2017.

\bibitem{bib:frw03}
F. Flandoli, F. Russo, and J. Wolf.
\newblock Some {SDE}s with distributional drift. {I}. {G}eneral calculus.
\newblock {\em Osaka J. Math.}, 40(2):493--542, 2003.

\bibitem{bib:frw04}
F. Flandoli, F. Russo, and J. Wolf.
\newblock Some {SDE}s with distributional drift. {II}. {L}yons-{Z}heng
 structure, {I}t\^o's formula and semimartingale characterization.
\newblock {\em Random Oper. Stochastic Equations}, 12(2):145--184, 2004.

\bibitem{Foll}
H. F\"ollmer.
\newblock Dirichlet processes.
\newblock In {\em Stochastic integrals ({P}roc. {S}ympos., {U}niv. {D}urham,
 {D}urham, 1980)}, volume 851 of {\em Lecture Notes in Math.}, pages 476--478.
 Springer, Berlin, 1981.

\bibitem{FH}
P. K. Friz and M. Hairer.
\newblock {\em A course on rough paths}.
\newblock Universitext. Springer, Cham, 2014.

\bibitem{GIP17}
M. Gubinelli, P. Imkeller, and N. Perkowski.
\newblock Paracontrolled distributions and singular {PDE}s.
\newblock {\em Forum Math. Pi}, 3(6), 2015.

\bibitem{GP15}
M. Gubinelli and N. Perkowski.
\newblock {\em Lectures on singular stochastic {PDE}s}, volume~29 of {\em
 Ensaios Matem\'aticos [Mathematical Surveys]}.
\newblock Sociedade Brasileira de Matem\'atica, Rio de Janeiro, 2015.


\bibitem{JS87}
J. Jacod and A. N. Shiryaev.
\newblock {\em Limit theorems for stochastic processes}, volume 288 of {\em
 Grundlehren der Mathematischen Wissenschaften [Fundamental Principles of
 Mathematical Sciences]}.
\newblock Springer-Verlag, Berlin, second edition, 2003.

\bibitem{bib:kims14}
P. Kim and R. Song.
\newblock Stable process with singular drift.
\newblock {\em Stochastic Process. Appl.}, 124(7):2479--2516, 2014.

\bibitem{kr_rock05}
N. V. Krylov and M. R\"ockner.
\newblock Strong solutions of stochastic equations with singular time dependent
 drift.
\newblock {\em Probab. Theory Related Fields}, 131(2):154--196, 2005.


\bibitem{K15}
A. Kulik
\newblock On weak uniqueness and distributional properties of a solution to an SDE with $\alpha$-stable noise
\newblock {\em Preprint : \href{https://arxiv.org/abs/1511.00106}{https://arxiv.org/abs/1511.00106}}, 2015


\bibitem{K14}
T. G. Kurtz.
\newblock Weak and strong solutions of general stochastic models.
\newblock {\em Electron. Commun. Probab.}, 19:no. 58, 16, 2014.

\bibitem{KP91}
T. G. Kurtz and P. Protter.
\newblock Weak limit theorems for stochastic integrals and stochastic
 differential equations.
\newblock {\em Ann. Probab.}, 19(3):1035--1070, 1991.

\bibitem{P15}
N. Perkowski.
\newblock Paracontrolled distributions and singular diffusions.
\newblock {\em Preprint : \href{http://www.mathematik.hu-berlin.de/~perkowsk/files/teaching/paracontrolled-bonn.pdf}{http://www.mathematik.hu-berlin.de/$\sim$perkowsk/files/teaching/paracontrolled-bonn.pdf}},
 2017.

\bibitem{Pr12}
E. Priola.
\newblock Pathwise uniqueness for singular {SDE}s driven by stable processes.
\newblock {\em Osaka J. Math.}, 49(2):421--447, 2012.

\bibitem{Pr18}
E. Priola.
\newblock Davie's type uniqueness for a class of sdes with jumps.
\newblock {\em Preprint : \href{https://arxiv.org/abs/1509.07448}{https://arxiv.org/abs/1509.07448}}, 2015.

\bibitem{Sato99}
K. Sato.
\newblock {\em L\'evy processes and infinitely divisible distributions},
 volume~68 of {\em Cambridge Studies in Advanced Mathematics}.
\newblock Cambridge University Press, Cambridge, 2013.
\newblock Translated from the 1990 Japanese original, Revised edition of the
 1999 English translation.

\bibitem{bib:ttw74}
H. Tanaka, M. Tsuchiya, and S. Watanabe.
\newblock Perturbation of drift-type for {L}\'evy processes.
\newblock {\em J. Math. Kyoto Univ.}, 14:73--92, 1974.

\bibitem{T}
H. Triebel.
\newblock {\em Theory of function spaces}.
\newblock Modern Birkh\"auser Classics. Birkh\"auser/Springer Basel AG, Basel,
 2010.
\newblock Reprint of 1983 edition, Also published in 1983 by
 Birkh\"auser Verlag.

\bibitem{ver80}
A. Ju. Veretennikov.
\newblock Strong solutions and explicit formulas for solutions of stochastic
 integral equations.
\newblock {\em Mat. Sb. (N.S.)}, 111(153)(3):434--452, 480, 1980.

\bibitem{bib:zz17}
X. Zhang and G. Zhao.
\newblock Heat kernel and ergodicity of sdes with distributional drifts.
\newblock {\em Preprint : \href{https://arxiv.org/abs/1710.10537v1}{https://arxiv.org/abs/1710.10537v1}}, 2017.

\bibitem{zvonkin74}
A. K. Zvonkin. 
\newblock A transformation of the phase space of a diffusion process that will
 remove the drift.
\newblock {\em Mat. Sb. (N.S.)}, 93(135):129--149, 152, 1974.

\end{thebibliography}


\end{document}